\documentclass[reqno,12pt]{amsart}
\usepackage[colorlinks=true, linkcolor=blue, citecolor=blue]{hyperref}

\usepackage{amssymb}
\usepackage{amsmath, graphicx, rotating}
\usepackage{color}
\usepackage{soul}
\usepackage[dvipsnames]{xcolor}

\usepackage{ifthen}
\usepackage{xkeyval}
\usepackage{todonotes}
\usepackage[T1]{fontenc}
\usepackage{lmodern}
\usepackage[english]{babel}

\usepackage{ upgreek }
\usepackage{stmaryrd}
\SetSymbolFont{stmry}{bold}{U}{stmry}{m}{n}
\usepackage{amsthm}
\usepackage{float}

\usepackage{ bbm }
\usepackage{ stmaryrd }
\usepackage{ mathrsfs }
\usepackage{ frcursive }
\usepackage{ comment }

\usepackage{pgf, tikz}
\usetikzlibrary{shapes}
\usepackage{varioref}
\usepackage{enumitem}

\definecolor{rouge}{rgb}{0.7,0.00,0.00}
\definecolor{vert}{rgb}{0.00,0.5,0.00}
\definecolor{bleu}{rgb}{0.00,0.00,0.8}
\usepackage[margin=1in]{geometry}
\newtheorem{theorem}{Theorem}[section]
\newtheorem*{theorem*}{Theorem}
\newtheorem{lemma}[theorem]{Lemma}

\newtheorem{proposition}[theorem]{Proposition}

\labelformat{hypothesis}{\textbf{M\kern-0.1mm#1}}

\newtheorem{condition}{Condition}

\theoremstyle{definition}
\newtheorem{example}[theorem]{Example}
\newtheorem{remark}[theorem]{Remark}

\def \eref#1{\hbox{(\ref{#1})}}

\numberwithin{equation}{section}

\def\geq{\geqslant}
\def\leq{\leqslant}

\def\RR{\mathbb{R}}

\def\EE{\mathbb{E}}

\def\vare{{\varepsilon}}
\def \eref#1{\hbox{(\ref{#1})}}

\def\EE{\mathbb{ E}}

\def\ep{{\epsilon}}

\begin{document}

\title[Strong and weak convergent rates for SDEs ]
{Strong and weak convergence rates for slow-fast stochastic differential equations driven by $\alpha$-stable process}

\author{Xiaobin Sun}
\curraddr[Sun, X.]{ School of Mathematics and Statistics, RIMS, Jiangsu Normal University, Xuzhou, 221116, China}
\email{xbsun@jsnu.edu.cn}

\author{Longjie Xie}
\curraddr[Xie, L.]{ School of Mathematics and Statistics, RIMS, Jiangsu Normal University, Xuzhou, 221116, China}
\email{longjiexie@jsnu.edu.cn}

\author{Yingchao Xie}
\curraddr[Xie, Y.]{ School of Mathematics and Statistics, RIMS, Jiangsu Normal University, Xuzhou, 221116, China}
\email{ycxie@jsnu.edu.cn}

\begin{abstract}
In this paper, we study the averaging principle for a class of stochastic differential equations
driven by $\alpha$-stable processes with slow and fast time-scales, where $\alpha\in(1,2)$.
We prove that the strong and weak convergence order are $1-1/\alpha$ and $1$ respectively.
We show, by a simple example, that  $1-1/\alpha$ is the optimal strong convergence rate.
\end{abstract}


\maketitle

\section{Introduction}

Multiscale models involving ``slow" and ``fast" components appear naturally in various fields,
such as  nonlinear oscillations, chemical kinetics, biology, climate dynamics, etc, see, e.g., \cite{BR, ELV,  HKW, PS} and the references therein.
The averaging principle of multiscale models describes the asymptotic behavior of the slow components as the scale parameter $\ep\to 0$.
In \cite{K1}, Khasminskii considered a class of multiscale stochastic differential equations (SDEs for short) driven by Wiener noise, i.e.,
\begin{eqnarray*}\left\{\begin{array}{l}
\displaystyle
d X^{\ep}_t = A(X^{\ep}_{t}, Y^{\ep}_{t})dt+ d W_t,\quad X^{\ep}_0=x\in\RR^{d}, \nonumber\\
\displaystyle d Y^{\ep}_t =\frac{1}{\ep}B(X^{\ep}_{t},  Y^{\ep}_{t})dt+\frac{1}{\sqrt{\ep}}d W_t,\quad Y^{\ep}_0=y\in\RR^{d},
\end{array}\right.\nonumber
\end{eqnarray*}
where $W_t$ is a $d$-dimensional Brownian motion, $\ep$ is a small and positive parameter which describes the ratio of the time scale between the slow component $X^{\ep}_t$ and fast component $Y^{\ep}_t$, the coefficients $A, B:\RR^{d} \times \RR^{d} \rightarrow \RR^{d}$  are assumed to be Lipschitz continuous. It is assumed
that there exists a map $\bar{A}:\RR^{d }\rightarrow \RR^{d }$   such that
$$
\left|\frac{1}{T}\int^T_0 \EE A(x,Y^{x,y}_t)dt-\bar{A}(x)\right|\leq \alpha(T)(1+|x|^2)
$$
for some function $\alpha(T)$ vanishing as $T$ goes to infinity, where $Y^{x,y}_t$ is the unique solution of the corresponding frozen equation
$$
d Y^{x,y}_t =B(x, Y^{x,y}_{t})dt+d W_t,\quad Y^{x,y}_0=y\in\RR^{d }.
$$
Then Khasminskii proved that the slow component $X^{\ep}_t$ converges to $\bar{X}_t$ in  probability, where $\bar{X}_t$ is the unique solution of the corresponding averaged equation,
$$
d \bar{X}_t =\bar A(\bar{X})dt+d W_t,\quad \bar{X}_0=x\in\RR^{d}.
$$
Since this pioneering work, many people studied the averaging principle for  stochastic systems driven by Wiener noise, see e.g. \cite{K1,KY2,KY,Ki,Ki2,L1,V0,LRSX1} for the averaging principle of SDEs and \cite{C1,C2,CF,DSXZ,FL,FLL,FWLL,GP1,GP2,GP3,GP4,WR} for the averaging principle of stochastic partial differential equations (SPDEs for short).

\vspace{1mm}
All the papers mentioned above considered stochastic systems with continuous paths. However, in many applications,  systems driven by discontinuous noises appear naturally. There have been many papers devoted to study the averaging principle for slow-fast stochastic systems driven by jump noises, see e.g. \cite{GD,L2,X,XML1,XML2,ZFWL}. But in these papers, the noises are assumed to have  second order moments in order to obtain the usual energy estimates. This excludes the important $\alpha$-stable noise with $\alpha\in (0,2)$. The discontinuity  and the heavy tail property   make the $\alpha$-stable noise a useful
driving process in models arising in physics, telecommunication networks, finance and other fields, see e.g. \cite{A,BYY,PZ} and the references therein for more  backgrounds. Slow-fast stochastic systems driven by $\alpha$-stable noises have attracted more attention recently. Zhang et al.\cite{ZCZCDL} studied data assimilation and parameter estimation for a multiscale stochastic system with $\alpha$-stable noise. Zulfiqar et al. \cite{ZYHD} studied slow manifolds of a slow-fast stochastic evolutionary system with stable L\'evy noise. Bao et al. \cite{BYY} studied the strong averaging principle for two-time scale SPDEs driven by $\alpha$-stable noise. In \cite{SZ} and \cite{CSS}, the first named authors
and his collaborators studied the strong averaging principle for stochastic Ginzburg-Landau equation and stochastic Burgers equations driven by $\alpha$-stable processes, respectively.

\vspace{1mm}
There are also many works devoted to studying the rate of convergence in the averaging principle. The main motivation comes from the well-known Heterogeneous Multi-scale Methods  used to approximate the slow component, see e.g. \cite{Br3, ELV}. Furthermore, the rate of convergence is also known to be very important for functional limit theorems in probability theory and homogenization, see e.g. \cite{KY,PV1,PV2,WR}.
Strong and weak convergence rates for slow-fast stochastic systems with Wiener noise have been studied extensively, e.g., see \cite{GKK, L1, L2, ZFWL, RSX} for the finite dimension case and \cite{B1, B2, DSXZ, FWLL} for the infinite dimension case. Usually, Khasminskii's time discretization technique is   used to study the strong convergence rate  while the method of asymptotic expansion of solutions of Kolmogorov equations
is  used to study the  weak convergence rate.
We mention that in \cite{BYY,CSS,SZ}, because of the time discretization method used therein, no satisfactory
convergence rates were obtained.

\vspace{1mm}
In this paper, we consider the following multiscale  SDEs
driven by $\alpha$-stable processes:
\begin{equation}\left\{\begin{array}{l}\label{Equation}
\displaystyle
d X^{\ep}_t = b(X^{\ep}_{t}, Y^{\ep}_{t})dt+d L^{1}_t,\quad X^{\ep}_0=x\in\RR^{d_1}, \\
\displaystyle d Y^{\ep}_t =\frac{1}{\ep}f(X^{\ep}_{t},  Y^{\ep}_{t})dt+\frac{1}{\ep^{1/\alpha}}d L^{2}_t,\quad Y^{\ep}_0=y\in\RR^{d_2},
\end{array}\right.
\end{equation}
where $\{L^{1}_t\}_{t\geq 0}$ and $\{L^{2}_t\}$ are independent $d_1$ and $d_2$ dimensional isotropic $\alpha$-stable processes  with $\alpha \in (1, 2)$,  $b: \RR^{d_1}\times\RR^{d_2} \rightarrow \RR^{d_1}$ and $f:\RR^{d_1}\times\RR^{d_2}\rightarrow \RR^{d_2}$ are Borel  functions  whose regularity conditions will be stated below.

\vspace{1mm}

It is well known that for systems driven by Wiener noise, the optimal strong  convergence rate is $1/2$ and the weak convergence rate is $1$. A natural   question is that: for systems driven by $\alpha$-stable noises, what are the optimal strong and weak
convergence rates? The purpose of this paper is to establish both the strong and the weak convergence rates
for the stochastic system \eref{Equation}. We will prove that the weak convergence order is still $1$, but the strong convergence order is $1-1/\alpha$. We show, by a simple example,
that the strong convergence rate $1-1/\alpha$ is optimal.

\vspace{1mm}
The main technique used in this paper is  the Poisson equation, which is inspired from \cite{B2} (see also \cite{PV1,PV2,RSX,RSX2}). In contrast to the classical Khasminskii's time discretization, it has great advantage in obtaining the rates of convergence.
More precisely, we need to study the  following Poisson equation in $\RR^{d_2}$:
\begin{align}
-\mathscr{L}_2(x,y)\Phi(x,y)=b(x,y)-\bar{b}(x),\quad y\in\RR^{d_2},  \label{PE1}
\end{align}
where $x\in\RR^{d_1}$ is a parameter and $\mathscr{L}_2(x,y)$ is the non-local operator defined by (\ref{L_2}) below. The main difficulty lies in analyzing the regularity of the solution   $\Phi(x,y)$ of (\ref{PE1}).
In the case of Wiener noise, one needs the $C^2$-regularity of $\Phi$  with respect to $x$
in order to apply It\^o's formula with respect to $X^{\ep}_t$.
However, in the case of $\alpha$-stable noise, $C^{\alpha+}$-regularity of $\Phi$ with respect to $x$ is enough for the application of It\^o's formula, where $\varphi\in C^{\alpha+}(\RR^{d_1})$ means there exists a constant $\gamma>\alpha$ such that $\varphi\in C^{\gamma}(\RR^{d_1})$ (the detailed definition is given below). Thus, in this paper  we will only assume that
$b$ has $C^{\alpha+}$-regularity with respect to the $x$ variable.

\vspace{1mm}

The organization of this paper is as follows. In the next section, we introduce some notation and state our main results. Section 3 is devoted to study  the regularity of the solution for the Poisson equation. The strong and weak convergence rates are proved in Subsections 4.1 and 4.2 respectively. Finally,  some basic properties for the $\alpha$-stable process as well as  some a priori estimates for the solutions of the system (\ref{Equation}), the frozen equation (\ref{FEQ}) and averaged equation (\ref{1.3}) are given in the Appendix.
Throughout this paper, $C$ and $C_T$ stand for constants whose value may change from line to line, and $C_T$ is used to emphasize  that the constant depends on $T$.

\section{Notations and Main results}\label{sec.prelim}

We first introduce some notation throughout this paper.
$\RR^d$ stands for the $d$-dimensional Euclidean space and $\mathbb{N}_{+}$ stands for the collection of all the positive integers. We will use $|\cdot|$ and $\langle\cdot, \cdot\rangle$
to denote the Euclidean norm and Euclidean inner product respectively. We use $\|\cdot\|$ to denote the matrix norm. For any $k\in \mathbb{N}_{+}$ and $\delta\in (0, 1)$, we define
\begin{align*}
C^k(\RR^d):=&\{u: \RR^d\to \RR: u \mbox{ and all its partial derivatives up to order } k \mbox{ are continuous}\},\\
C^k_b(\RR^d):=&\{u\in C^k(\RR^d): \mbox{ for } 1\le i\le k, \mbox{ the $i$-the partial derivatives of } u \mbox{ are bounded} \},\\
C^{k+\delta}_b(\RR^d):=&\{u\in C^k_b(\RR^d):\!\mbox{ all the $k$-th order partial derivatives of }
u \mbox{ are $\delta$-H\"{o}lder continuous}\}.
\end{align*}
For any $k\in\mathbb{N}_{+}, \delta\in (0,1)$, the spaces $C^{k}_b(\RR^d)$ and $C^{k+\delta}_b(\RR^d)$
when equipped with the usual norm $\|\cdot\|_{C^{k}_b}$ and $\|\cdot\|_{C^{k+\delta}_b}$,
are Banach spaces.

For $k_1,k_2\in\mathbb{N}_{+}$, $0\leq \delta_1,\delta_2< 1$ and a real-valued function
on $\RR^{d_1}\times\RR^{d_2}$, the notation $u\in C^{k_1+\delta_1,k_2+\delta_2}_b$
means that (i) for all $d_1$-tuple $\beta$ and $d_2$-tuple $\gamma$ with $0\le |\beta|\le k_1$
$0\le |\gamma|\le k_2$ and $|\beta|+|\gamma|\ge 1$, the partial derivative $\partial^{\beta}_{x}\partial^{\gamma}_{y}u$ is bounded continuous; and (ii) for any $|\beta|= k_1$ and
$0\leq |\gamma|\leq 1$, $\partial^{\beta}_{x}\partial^{\gamma}_yu$ is $\delta_1$-H\"{o}lder continuous with respect to $x$ with index $\delta_1$  uniformly in $y$, and for any $|\gamma|=k_2$, $\partial^{\gamma}_{y}u$ is $\delta_2$-H\"{o}lder continuous with respect to $y$ with index $\delta_2$  uniformly in $x$.



\vskip 0.2cm

Our first result is as follows.

\begin{theorem}(\textbf {Strong convergence})\label{main result 1}
Suppose that $b\in C^{1+\gamma,2+\delta}_b$ and $f\in C^{1+\gamma,2+\gamma}_b$ with some $\gamma\in(\alpha-1, 1)$ and $\delta\in(0,1)$,
and that there exists $\beta>0$ such that for any $x\in\RR^{d_1}$, $y_1,y_2\in\RR^{d_2}$,
\begin{eqnarray}
\sup_{x\in\RR^{d_1}}|f(x,0)|<\infty,\quad \langle f(x, y_1)-f(x, y_2), y_1-y_2\rangle \leq -\beta|y_1-y_2|^2.\label{sm}
\end{eqnarray}
Then for any initial value $(x,y)\in\RR^{d_1}\times\RR^{d_2}$, $T>0$ and $p\in [1,\alpha)$, we have
\begin{eqnarray}
\mathbb{E}\left(\sup_{t\in[0,T]}|X_{t}^{\ep}-\bar{X}_{t}|^p\right)\leq C\ep^{p(1-1/\alpha)}, \label{R1}
\end{eqnarray}
where $C$ is a positive constant
depending on $p, T, |x|, |y|$,
and $\bar{X}$ is the solution of the following averaged equation:
\begin{equation}
d\bar{X}_{t}=\bar{b}(\bar{X}_t)dt+d L^{1}_t,\quad\bar{X}_{0}=x, \label{1.3}
\end{equation}
where $\bar{b}(x)=\int_{\RR^{d_2}}b(x,y)\mu^{x}(dy)$, and $\mu^{x}$ denotes the unique invariant measure for the transition semigroup of the corresponding frozen equation
\begin{equation}\label{FEQ}
dY_{t}=f(x,Y_{t})dt+d L_{t}^{2},\quad Y_{0}=y.
\end{equation}
\end{theorem}

\begin{example}
The estimate \eref{R1} implies that the strong convergent order is $1-1/\alpha$.
Following \cite[Section 2]{L1}, we  show via an simple example that $1-1/\alpha$ is the optimal order.  Consider
\begin{equation}\left\{\begin{array}{l}
\displaystyle
dX^{\ep}_t=Y^{\ep}_t dt+dL^1_t,\quad X^{\ep}_0=x\in\mathbb{R}, \nonumber\\
\displaystyle dY^{\ep}_t=-\frac{1}{\ep}Y^{\ep}_tdt+\frac{1}{\ep^{1/\alpha}}dL^2_t,\quad Y^{\varepsilon}_0=0\in\mathbb{R}, \nonumber\\
\end{array}\right.
\end{equation}
where $\{L^{1}_t\}_{t\geq 0}$ and $\{L^{2}_t\}_{t\geq 0}$ are independent $1$-dimensional symmetric $\alpha$-stable process.
The solution of the equation above is given by
\begin{equation}\left\{\begin{array}{l}
\displaystyle
X^{\varepsilon}_t=x+\int^t_0 Y^{\ep}_s ds+L^1_t, \nonumber\\
\displaystyle Y^{\varepsilon}_t=\frac{1}{\ep^{1/\alpha}}\int^{t}_0 e^{-(t-s)/\ep}dL^2_s. \nonumber\\
\end{array}\right.
\end{equation}
Note that the corresponding frozen equation is
$$
dY_t=-Y_tdt+dL^2_t,\quad Y_0=0,
$$
which has a unique solution $Y_t=\int^{t}_0 e^{-(t-s)}dL^2_s$. Then it is easy to prove $\{Y_t, t\geq 0\}$ has a unique invariant measure with mean zero.
Thus, the corresponding averaged equation is given by
$$\bar{X}_t=x+L^1_t.$$
As a result, we have for $0<p<\alpha$,
$$\EE\left|X^{\ep}_t-\bar{X}_t\right|^p=\EE\left|\int^t_0 Y^{\ep}_s ds\right|^p.$$
Put $Z^{\ep}_t:=\int^t_0 Y^{\ep}_s ds$, then it is easy to see that
$$Z^{\ep}_t=\frac{1}{\ep^{1/\alpha}}\int^t_0 \int^s_0 e^{-\frac{1}{\ep}(s-r)}dL^2_r ds=\frac{1}{\ep^{1/\alpha}}\int^t_0 \left[ \int^t_r e^{-\frac{1}{\ep}(s-r)}ds \right]dL^2_r.$$
Refer to \cite[(2.4)]{PZ}, for any continuous function $f: [0, t]\rightarrow R$, we have
$$
\EE\left[e^{ih\int^t_0f_s dL^2_s}\right]=\exp\left\{-\int^t_0\psi (f_s h)ds\right\},\quad h\in \mathbb{R},
$$
where $\psi(x)=C_{\alpha}|x|^{\alpha}$. As a result,  the characteristic function of $Z^{\ep}_t$ is given by
\begin{eqnarray*}
\varphi_{Z^{\ep}_t}(h):=\!\!\!\!\!\!\!\!&&\EE\left( e^{ihZ^{\ep}_t}\right)=\exp\left\{-\int^t_0 \psi\left(\frac{h}{\ep^{1/\alpha}}\int^t_r e^{-\frac{1}{\ep}(s-r)}ds \right)dr\right\}\\
=\!\!\!\!\!\!\!\!&&\exp\left\{-\int^t_0 C_{\alpha}(1-e^{-\frac{r}{\ep}})^{\alpha}dr \left(\ep^{1-1/\alpha}\right)^\alpha |h|^{\alpha}\right\},\quad h\in\RR.
\end{eqnarray*}
Refer to \cite[(3.2)]{PZ}, for any symmetric real $\alpha$-stable random variable $X$ has the the characteristic function:
\begin{eqnarray*}
\EE\left[e^{ihX}\right]=\exp\{-\sigma^{\alpha}|h|^{\alpha}\},
\end{eqnarray*}
for some $\sigma\geq 0$, then we have
\begin{eqnarray*}
\EE|X|^p=C_{\alpha, p}\sigma^p, \quad 0<p<\alpha.
\end{eqnarray*}
Thus, it is easy to see that
\begin{eqnarray*}
\EE\left|\int^t_0 Y^{\ep}_s ds\right|^p=\!\!\!\!\!\!\!\!&&C_{\alpha,p}\left[\int^t_0 \left(1-e^{-\frac{r}{\ep}}\right)^{\alpha}dr\right]^{p/\alpha} \left(\ep^{1-1/\alpha}\right)^p,
\end{eqnarray*}
which implies the desired result.
\end{example}

The following is our second main result about the weak convergence rate.
\begin{theorem}(\textbf{Weak convergence})\label{main result 2}
Suppose that the assumptions in Theorem \ref{main result 1} holds. Assume further that
$b$ is uniformly bounded and $b, f\in C^{2+\gamma,2+\gamma}_b$  with $\gamma\in (\alpha-1, 1)$.
Then for any $\phi\in C^{2+\gamma}_b$, initial value $(x,y)\in\RR^{d_1}\times\RR^{d_2}$ and $T>0$, we have
\begin{eqnarray}
\sup_{t\in[0, T]}\left|\mathbb{E}\phi(X_{t}^{\ep})-\EE \phi(\bar{X}_{t})\right|\leq C\ep, \label{R3}
\end{eqnarray}
where $C$ is a positive constant
depending on $T$,$\|\phi\|_{C^{2+\gamma}_b}$, $|x|$ and $|y|$,
and $\bar{X}$ is the solution of the averaged equation \eref{1.3}.
\end{theorem}

\section{Non-local Poisson equation}

This section is devoted to study the Poisson equation. In Subsection 3.1, we prove the exponential ergodicity for the transition semigroup of the frozen equation (\ref{FEQ}). In Subsection 3.2, we prove the well-posedness as well as the regularities of the solution for the Poisson equation (\ref{PE1}).

\subsection{The frozen equation}
Let  $Y_{t}^{x,y}$ satisfies the frozen equation
\begin{equation*}
dY_{t}=f(x,Y_{t})dt+d L_{t}^{2},\quad Y_{0}=y\in \RR^{d_2}.
\end{equation*}
Some moment estimates for $Y_t^{x,y}$ are collected in Lemma \ref{L3.2} in Appendix.
Note that for any $\vare>0$, define $\{\tilde{L}^2_{t}:=\frac{1}{\vare^{1/\alpha}}L^2_{t\vare}, t\geq 0\}$, then $\tilde{L}^2_{t}$ is again
an $\alpha$-stable process and
\begin{eqnarray}
Y^{\vare}_{t\vare}=y+\frac{1}{\vare}\int^{t\vare}_0 f(X^{\vare}_s, Y^{\vare}_s)ds+\frac{1}{\vare^{1/\alpha}}L^2_{t\vare}
=y+\int^{t}_0f(X^{\vare}_{s\vare}, Y^{\vare}_{s\vare})ds+\tilde{L}^2_{t}.\label{R3.11}
\end{eqnarray}
This explains the scaling $\vare^{1/\alpha}$ in the fast component of (\ref{Equation}).
Here, we prove the following results, which will be used below.

\begin{lemma}\label{L3.6}
Assume that $f\in C^{1,1}_b$ and condition \eref{sm} holds. Then for any $t\geq 0$, $x_i\in\RR^{d_1}$, and $y_i\in\RR^{d_2}$, $i=1,2$, we have
\begin{eqnarray*}
\left|Y^{x_1,y_1}_t-Y^{x_2,y_2}_t\right|\leq e^{-\frac{\beta t}{2}}|y_1-y_2|+C|x_1-x_2|,
\end{eqnarray*}
where $C$ is a constant independent of $t$.
\end{lemma}
\begin{proof}
Note that
\begin{eqnarray*}
d\left(Y^{x_1,y_1}_t-Y^{x_2,y_2}_t\right)=\!\!\!\!\!\!\!\!&&\left[f(x_1, Y^{x_1,,y_1}_t)-f(x_2, Y^{x_2,y_2}_t)\right]dt,\quad Y^{x_1,y_1}_0-Y^{x_2,y_2}_0=y_1-y_2.
\end{eqnarray*}
Multiplying  both sides by $2(Y^{x_1,y_1}_t-Y^{x_2,y_2}_t)$, we obtain
\begin{eqnarray*}
\frac{d}{dt}\left|Y^{x_1,y_1}_t-Y^{x_2,y_2}_t\right|^2=\!\!\!\!\!\!\!\!&&2\langle f( x_1,Y^{x_1,y_1}_t)-f(x_2, Y^{x_2,y_2}_t), Y^{x_1,y_1}_t-Y^{x_2,y_2}_t\rangle.
\end{eqnarray*}
Then by condition \eref{sm} and Young's inequality, we get
\begin{eqnarray*}
\frac{d}{dt}|Y^{x_1,y_1}_t-Y^{x_2,y_2}_t|^2=\!\!\!\!\!\!\!\!&&2\left\langle f(x_1, Y^{x_1,y_1}_t)-f(x_1,Y^{x_2,y_2}_t), Y^{x_1,y_1}_t-Y^{x_2,y_2}_t\right\rangle\\
&&+2\left\langle f(x_1, Y^{x_2,y_2}_t)-f(x_2,Y^{x_2,y_2}_t), Y^{x_1,y_1}_t-Y^{x_2,y_2}_t\right\rangle\\
\leq\!\!\!\!\!\!\!\!&& -2\beta\left|Y^{x_1,y_1}_t-Y^{x_2,y_2}_t\right|^2+C|x_1-x_2|\left|Y^{x_1,y_1}_t-Y^{x_2,y_2}_t\right|\\
\leq\!\!\!\!\!\!\!\!&& -\beta\left|Y^{x_1,y_1}_t-Y^{x_2,y_2}_t\right|^2+C|x_1-x_2|^2.
\end{eqnarray*}
As a result of the comparison theorem, we have that for any $t\geq 0$,
\begin{eqnarray*}
\left|Y^{x_1,y_1}_t-Y^{x_2,y_2}_t\right|^2\leq e^{-\beta t}|y_1-y_2|^2+C|x_1-x_2|^2.
\end{eqnarray*}
The proof is complete.
\end{proof}

Let $\{P^{x}_t\}_{t\geq 0}$ be the transition semigroup of $Y_{t}^{x,y}$, i.e., for any bounded measurable function $\varphi:\RR^{d_2}\rightarrow \mathbb{R}$,
$$
P^{x}_t\varphi(y):=\EE\varphi(Y_{t}^{x,y}), \quad y\in\RR^{d_2}, t\geq 0.
$$
Then condition \eref{sm} ensures that $P^{x}_t$ has a unique invariant measure $\mu^{x}$ (see e.g. \cite[Theorem 1.1]{W}). Moreover, in view of \eref{FEq0} in the Appendix, we have for any $p\in (0, \alpha)$,  $\sup_{x\in\RR^{d_1}}\int_{\RR^{d_2}}|z|^p\mu^x(dz)<\infty$. The following result will play
an important role below.

\begin{proposition}\label{Ergodicity} Assume that $f(x,\cdot)\in C^1_b$ and condition \eref{sm} holds. Then for any function $g\in C^1_b$, there exists a positive constant $C$ such that for any $t\geq 0$ and $y\in\RR^{d_2}$,
\begin{eqnarray}
\sup_{x\in\RR^{d_1}}\left| P^x_tg(y)-\mu^x(g)\right|\leq\!\!\!\!\!\!\!\!&& C\|g\|_1 e^{-\frac{\beta t}{2}}(1+|y|),\label{ergodicity1}
\end{eqnarray}
where $\|g\|_1:=\sup_{x\neq y\in\RR^{d_2}}\frac{|g(x)-g(y)|}{|x-y|}$.
\end{proposition}
\begin{proof}
By the definition of invariant measure and Lemma \ref{L3.6},  for any $t\geq 0$ we have
\begin{eqnarray*}
\left| \EE g(Y^{x,y}_t)-\mu^x(g)\right|=\!\!\!\!\!\!\!\!&&\left| \EE g(Y^{x,y}_t)-\int_{\RR^{d_2}}g(z)\mu^{x}(dz)\right|\\
\leq\!\!\!\!\!\!\!\!&& \left|\int_{\RR^{d_2}}\left[\EE g(Y^{x,y}_t)-\EE g(Y^{x,z}_t)\right]\mu^{x}(dz)\right|\\
\leq\!\!\!\!\!\!\!\!&& \|g\|_1\int_{\RR^{d_2}} \EE\left| Y^{x,y}_t-Y^{x,z}_t\right|\mu^{x}(dz)\\
\leq\!\!\!\!\!\!\!\!&& \|g\|_1 e^{-\frac{\beta t}{2}}\int_{\RR^{d_2}}|y-z|\mu^{x}(dz)\\
\leq\!\!\!\!\!\!\!\!&& \leq C\|g\|_1 e^{-\frac{\beta t}{2}}(1+|y|).
\end{eqnarray*}
The proof is complete.
\end{proof}

\subsection{Poisson equaiton}

Let $\mathscr{L}_{2}(x,y)$ be the generator of $Y_t^{x,y}$, i.e.,
\begin{eqnarray}\label{L_2}
\mathscr{L}_{2}(x,y)\Phi(x,y):=\!\!\!\!\!\!\!\!&&-(-\Delta_y)^{\alpha/2}\Phi(x,y)+\langle f(x,y), \nabla_y \Phi(x,y)\rangle.
\end{eqnarray}
 The following is the main result of this subsection.

\begin{proposition}\label{P3.6}
Suppose that the assumptions in Theorem \ref{main result 1} hold.
Define
\begin{eqnarray}
\Phi(x,y):=\int^{\infty}_{0}\left[\EE b(x,Y^{x,y}_t)-\bar{b}(x)\right]dt.\label{SPE}
\end{eqnarray}
Then $\Phi(x,y)$ is a solution of the Poisson equation \eref{PE1}. Moreover, we have $\Phi(\cdot,y)\in C^{1}(\RR^{d_1},\RR^{d_1})$, $\Phi(x,\cdot)\in C^{2}(\RR^{d_2}, \RR^{d_1})$, and there exists $C>0$ such that
\begin{eqnarray}
\sup_{x\in\RR^{d_1}}|\Phi(x,y)|\leq C(1+|y|),\quad\sup_{x\in\RR^{d_1},y\in\RR^{d_2}}\|\nabla_y \Phi(x,y)\|\leq C,\label{E1}
\end{eqnarray}
and for any $\theta\in(0,1]$, there exists $C_{\theta}>0$ such that for any $x_1,x_2\in\RR^{d_1}, y\in\RR^{d_2}$,
\begin{eqnarray}
\sup_{x\in\RR^{d_1}}\|\nabla_x \Phi(x,y)\|\leq C_{\theta}(1+|y|^{\theta}), \label{E2}
\end{eqnarray}
\begin{eqnarray}
\| \nabla_x\Phi(x_1, y)- \nabla_x\Phi(x_2, y)\|\leq C|x_1-x_2|^{\gamma}(1+|x_1-x_2|^{1-\gamma})(1+|y|),\label{E3}
\end{eqnarray}
where $\gamma\in (\alpha-1, 1)$ is the constant in Theorem \ref{main result 1}.
\end{proposition}

\begin{proof}
The assertion that \eref{SPE} is a solution of the Poisson equation \eref{PE1} follows by It\^o's formula. Moreover, by straightforward computation, we can see that $\Phi(\cdot,y)\in C^{1}(\RR^{d_1},\RR^{d_1})$ and $\Phi(x,\cdot)\in C^{2}(\RR^{d_2}, \RR^{d_1})$. Below, we mainly focus on the regularity estimates (\ref{E1})-(\ref{E3}).
By Proposition \ref{Ergodicity}, we have
\begin{eqnarray*}
|\Phi(x,y)|\leq\!\!\!\!\!\!\!\!&&\int^{\infty}_{0}| \EE[b(x,Y^{x,y}_t)]-\bar{b}(x)|dt\\
\leq\!\!\!\!\!\!\!\!&& C(1+|y|)\int^{\infty}_{0}e^{-\frac{\beta t}{2}}dt\leq C_{\beta}(1+|y|),
\end{eqnarray*}
which implies the first estimate in \eref{E1}. Note that
$$\nabla_y \Phi(x,y)=\int^{\infty}_0 \EE[\nabla_y b(x,Y^{x,y}_t)\cdot\nabla_y Y^{x,y}_t]dt,$$
where $\nabla_y Y^{x,y}_t$ satisfies
 \begin{equation}\left\{\begin{array}{l}\label{partial y}
\displaystyle
d\nabla_y Y^{x,y}_t=\nabla_y f(x,Y^{x,y}_{t})\cdot\nabla_y Y^{x,y}_tdt, \vspace{2mm}\\
\nabla_y Y^{x,y}_t=I.\\
\end{array}\right.
\end{equation}
Then by Lemma \ref{L3.6}, we have
\begin{eqnarray}
\sup_{x\in\RR^{d_1},y\in\RR^{d_2}}\|\nabla_y Y^{x,y}_t\|\leq Ce^{-\frac{\beta t}{2}}.\label{partial yY}
\end{eqnarray}
Thus, by the boundedness of $\|\nabla_y b(x,y)\|$, there exists $C>0$ such that
$$\sup_{x\in\RR^{d_1},y\in\RR^{d_2}}\|\nabla_y \Phi(x,y)\|\leq C,$$
which implies the second estimate in \eref{E1}.

Now, we define
\begin{eqnarray*}
\tilde b_{t_0}(x, y, t):=\hat b(x,y, t)-\hat b(x, y, t+t_0),
\end{eqnarray*}
where $\hat b(x, y, t):=\EE b(x, Y^{x, y}_t)$.
We claim that for any $\theta\in (0,1]$, there exist $C_{\theta}>0$ and $\eta>0$ such that for any $t_0>0$, $ t> 0$, $x\in\RR^{d_1}$ and $y\in\RR^{d_2}$ , we have
\begin{eqnarray}
\left\|\nabla_x\tilde b_{t_0}(x,y,t)\right\|\leq C_{\theta} e^{-\eta t}(1+|y|^{\theta}),\label{E21}
\end{eqnarray}
\begin{eqnarray}
\left\|\nabla_x \tilde b_{t_0}(x_1,y,t)\!-\!\nabla_x \tilde b_{t_0}(x_2,y,t)\right\|\leq C e^{-\eta t}|x_1\!-\!x_2|^{\gamma}(1+|x_1-x_2|^{1-\gamma})(1+|y|)\label{E22}.
\end{eqnarray}
On the other hand, Proposition \ref{Ergodicity} implies that
$$
\lim_{t_0\to +\infty}\tilde b_{t_0}(x, y, t)=\EE[b(x,Y^{x,y}_t)]-\bar{b}(x)
$$
and thus
$$
\lim_{t_0\to +\infty} \nabla_x\tilde b_{t_0}(x, y, t)=\nabla_x\left[\EE b(x,Y^{x,y}_t)-\bar{b}(x)\right ].
$$
Consequently, we have estimates \eref{E2} and \eref{E3} hold. Below, we provide the proof of estimates \eref{E21}  and \eref{E22} separately.
\end{proof}

\begin{proof}[Proof of \eref{E21}]
By the Markov property, we write
\begin{eqnarray*}
\tilde b_{t_0}(x,y, t)=\!\!\!\!\!\!\!\!&& \hat b(x, y, t)-\EE b(x, Y^{x,y}_{t+t_0})\nonumber\\
=\!\!\!\!\!\!\!\!&& \hat b(x,y, t)-\EE \{\EE[b(x,Y^{x,y}_{t+t_0})|\mathscr{F}_{t_0}]\}=\hat b(x, y, t)-\EE \hat b(x, Y^{x,y}_{t_0},t).
\end{eqnarray*}
Thus
\begin{align}\label{5.8}
\nabla_x\tilde b_{t_0}(x,y, t)= \nabla_x \hat b(x, y, t)\!-\!\EE \nabla_x\hat b(x, Y^{x,y}_{t_0},t)-\!\! \EE \left[\nabla_y\hat b(x,Y^{x,y}_{t_0},t)\cdot \nabla_x Y^{x,y}_{t_0} \right],
\end{align}
where $\nabla_x Y^{x,y}_t$ satisfies
 \begin{equation*}
d\nabla_x Y^{x,y}_t=\nabla_x f(x,Y^{x,y}_{t})dt+\nabla_y f(x,Y^{x,y}_{t})\cdot\nabla_x Y^{x,y}_tdt,\quad
\nabla_x Y^{x,y}_0=0.
\end{equation*}
Obviously, Lemma \ref{L3.6} implies
\begin{eqnarray}
\sup_{t\geq 0, x\in\RR^{d_1}, y\in\RR^{d_2},} \|\nabla_x Y^{x,y}_t\|\leq C. \label{S0}
\end{eqnarray}
Note that $\nabla_y\hat b(x, y, t)=\EE\left[\nabla_y b(x, Y^{x,y}_t)\cdot \nabla_y Y^{x,y}_t\right]$, combining this with \eref{partial yY} and the boundedness of $\|\nabla_y b(x,y)\|$, we get
\begin{eqnarray}
\sup_{x\in\RR^{d_1},y\in\RR^{d_2}}\|\nabla_y\hat b(x, y, t)\|\leq\!\!\!\!\!\!\!\!&&C e^{-\frac{\beta t}{2}}.\label{S1}
\end{eqnarray}
Next, if we can show that for any $\theta\in (0,1]$, there exists $C_{\theta}>0$ such that for any $t \geq 0$, $x\in\RR^{d_1}$ and $y_1,y_2\in\RR^{d_2}$,
\begin{eqnarray}
\left\|\nabla_x\hat b(x,y_1, t)-\nabla_x\hat b(x, y_2, t)\right\|\leq C_{\theta} e^{-\frac{\beta\theta t}{2}}|y_1-y_2|^{\theta}.\label{S2}
\end{eqnarray}
Then, by \eref{5.8}, estimates \eref{S0}-\eref{S2} and \eref{FEq0}, we have
\begin{eqnarray*}
\left\|\nabla_x\tilde b_{t_0}( x, y, t)\right\|\leq  C_{\theta}e^{-\frac{\beta \theta t}{2}} \EE|y-Y^{x,y}_{t_0}|^{\theta}+Ce^{-\frac{\beta t}{2}} \leq C_{\theta} e^{-\frac{\beta\theta t}{2}}(1+|y|^{\theta}),
\end{eqnarray*}
which proves \eref{E21}.

Now, we proceed to prove \eref{S2}. Indeed,
\begin{align*}
&\left\|\nabla_x\hat b(x, y_1, t)-\nabla_x\hat b(x, y_2, t)\right\|=\left\| \nabla_x \EE b(x, Y^{x, y_1}_t)-\nabla_x\EE b(x,Y^{x,y_2}_t)\right\|\\
&\leq \EE \left\|\nabla_x b(x, Y^{x,y_1}_t)-\nabla_x b(x, Y^{x,y_2}_t)\right\|\\
&\quad+\EE \left\|\nabla_y b(x, Y^{x,y_1}_t)\cdot\nabla_x Y^{x,y_1}_t-\nabla_y b(x,  Y^{x,y_2}_t)\cdot\nabla_x Y^{x,y_2}_t\right\|\\
&\leq \EE \left\|\nabla_x b(x,Y^{x,y_1}_t)- \nabla_x b(x,  Y^{x,y_2}_t)\right\|\\
&\quad+\EE \left\|\nabla_y b(x, Y^{x,y_1}_t)\cdot\nabla_x Y^{x,y_1}_t-\nabla_y b(x, Y^{x,y_2}_t)\cdot\nabla_x Y^{x,y_1}_t\right\|\\
&\quad+\EE \left\|\nabla_y b(x, Y^{x,y_2}_t)\cdot\nabla_x Y^{x,y_1}_t-\nabla_y b(x, Y^{x,y_2}_t)\cdot\nabla_x Y^{x,y_2}_t\right\|=:\sum^{3}_{i=1}S_i.
\end{align*}
By
the boundedness of $\|\nabla_x b(x,y)\|$, $\|\nabla_y\nabla_x b(x,y)\|$ and $\|\partial^2_y b(x,y)\|$, we have for any $\theta\in (0,1]$,
\begin{eqnarray}
S_1\leq C_{\theta}\EE|Y^{x,y_1}_t-Y^{x,y_2}_t|^{\theta}\leq C_{\theta} e^{-\frac{\beta\theta t}{2}}|y_1-y_2|^{\theta},\label{S_1}
\end{eqnarray}
and by the boundedness of $\|\nabla^2_y b(x,y)\|$ and Lemma \ref{L3.6},
\begin{eqnarray}
S_2\leq C_{\theta} \EE|Y^{x,y_1}_t-Y^{x,y_2}_t|^{\theta}\leq C_{\theta} e^{-\frac{\beta \theta t}{2}}|y_1-y_2|^{\theta}.\label{S_2}
\end{eqnarray}
For the term $S_3$, by  the assumption $f\in C^{1+\gamma,2}_b$, condition \eref{sm} and a straightforward computation, we have that for any $x_1,x_2\in\RR^{d_1}$, $y_1,y_2\in\RR^{d_2}$,
\begin{eqnarray}
\left\|\nabla_x Y^{x_1,y_1}_t-\nabla_x Y^{x_2,y_2}_t\right\|^2\leq\!\!\!\!\!\!\!\!&& C(|x_1-x_2|^{2\gamma}+|x_1-x_2|^{2})+Ce^{-\frac{\beta t}{2}}|y_1-y_2|^2.\label{partial xy Y}
\end{eqnarray}
Then by \eref{S0} and \eref{partial xy Y}, it is easy to see
\begin{eqnarray}
S_3\leq C_{\theta}\EE\|\nabla_x Y^{x,y_1}_t-\nabla_x Y^{x,y_2}_t\|^{\theta}
\leq C_{\theta} e^{-\frac{\beta \theta t}{4}}|y_1-y_2|^{\theta}.\label{S_3}
\end{eqnarray}
Now \eref{S2} follows directly from \eref{S_1}-\eref{S_3}.
\end{proof}

\begin{proof}[Proof of \eref{E22}]
Recall that
\begin{eqnarray*}
\nabla_x\tilde b_{t_0}(x,y, t)=\!\!\!\!\!\!\!\!&& \nabla_x \hat b(x, y, t)-\EE \nabla_x\hat b(x, Y^{x,y}_{t_0},t)- \EE \left[\nabla_y\hat b(x,Y^{x,y}_{t_0},t)\cdot\nabla_x Y^{x,y}_{t_0} \right].
\end{eqnarray*}
Then we get for any $x_1,x_2\in\RR^{d_1}$, $y\in\RR^{d_2}$ and $t,t_0>0$,
\begin{eqnarray*}
&&\left\|\nabla_x\tilde b_{t_0}(x_1, y, t)-\nabla_x\tilde b_{t_0}(x_2, y, t)\right\|\\
=\!\!\!\!\!\!\!\!&& \left\|\nabla_x \hat b(x_1, y, t)-\EE \nabla_x\hat b(x_1, Y^{x_1,y}_{t_0},t)-\left[\nabla_x \hat b(x_2, y, t)-\EE \nabla_x\hat b(x_2, Y^{x_2,y}_{t_0},t)\right]\right.\nonumber\\
&&\left.-\EE \left[\nabla_y\hat b(x_1,Y^{x_1,y}_{t_0},t)\cdot\nabla_x Y^{x_1,y}_{t_0} \right]+\EE \left[\nabla_y\hat b(x_2,Y^{x_2,y}_{t_0},t)\cdot\nabla_x Y^{x_2,y}_{t_0} \right]\right\|\\
\leq\!\!\!\!\!\!\!\!&& \left\|\nabla_x \hat b(x_1, y, t)-\EE \nabla_x\hat b(x_1, Y^{x_1,y}_{t_0},t)-\left[\nabla_x \hat b(x_2, y, t)-\EE \nabla_x\hat b(x_2, Y^{x_1,y}_{t_0},t)\right]\right\|\nonumber\\
&&+\left\|\EE \nabla_x\hat b(x_2, Y^{x_2,y}_{t_0},t)-\EE \nabla_x\hat b(x_2, Y^{x_1,y}_{t_0},t)\right\|\\
&&+\left\|\EE \left[ \nabla_y\hat b(x_1,Y^{x_1,y}_{t_0},t)\cdot\nabla_x Y^{x_1,y}_{t_0}\right]-\EE \left[ \nabla_y\hat b(x_2,Y^{x_2,y}_{t_0},t) \cdot \nabla_x Y^{x_2,y}_{t_0} \right]\right\|\\
:=&&\sum^3_{i=1} Q_i.
\end{eqnarray*}
(i) For the term $Q_1$, recall that
\begin{eqnarray*}
\nabla_x\hat b(x,y, t)=\!\!\!\!\!\!\!\!&& \EE \left[\nabla_x b(x, Y^{x, y}_{t})\right]+\EE\left[ \nabla_yb(x,Y^{x,y}_{t})\cdot\nabla_xY^{x,y}_t \right],
\end{eqnarray*}
which implies
\begin{eqnarray*}
&&\left\|\nabla_x\hat b(x_1,y_1, t)-\nabla_x\hat b(x_1,y_2, t)-[\nabla_x\hat b(x_2,y_1, t)-\nabla_x\hat b(x_2,y_2, t)]\right\|\\
=\!\!\!\!\!\!\!\!&&\left\|\EE \left[\nabla_xb(x_1,Y^{x_1,y_1}_t)\right]+\EE \left[\nabla_yb(x_1,Y^{x_1,y_1}_t)\cdot \nabla_x Y^{x_1,y_1}_t\right]\right.\nonumber\\
&&-\EE \left[\nabla_xb(x_1,Y^{x_1,y_2}_t)\right]-\EE \left[\nabla_yb(x_1,Y^{x_1,y_2}_t)\cdot \nabla_x Y^{x_1,y_2}_t\right]\\
&&-\EE \left[\nabla_xb(x_2,Y^{x_2,y_1}_t)\right]-\EE \left[\nabla_yb(x_2,Y^{x_2,y_1}_t)\cdot \nabla_x Y^{x_2,y_1}_t\right]\\
&&+\left.\EE \left[\nabla_xb(x_2,Y^{x_2,y_2}_t)\right]+\EE \left[\nabla_yb(x_2,Y^{x_2,y_2}_t)\cdot \nabla_x Y^{x_2,y_2}_t\right]\right\|
\leq\sum^{3}_{i=1}Q_{1i},
\end{eqnarray*}
where
$$Q_{11}:=\left\|\EE \left[\nabla_xb(x_1,Y^{x_1,y_1}_t)-\nabla_xb(x_1,Y^{x_1,y_2}_t)-(\nabla_xb(x_2,Y^{x_1,y_1}_t)-\nabla_xb(x_2,Y^{x_1,y_2}_t))\right]\right\|,$$
$$Q_{12}:=\|\EE \left[\nabla_xb(x_2,Y^{x_1,y_1}_t)-\nabla_xb(x_2,Y^{x_2,y_1}_t)-(\nabla_xb(x_2,Y^{x_1,y_2}_t)-\nabla_xb(x_2,Y^{x_2,y_2}_t))\right],$$
\begin{eqnarray*}
&&Q_{13}:=\left\|\EE \left[\nabla_yb(x_1,Y^{x_1,y_1}_t)\cdot \nabla_x Y^{x_1,y_1}_t-\nabla_yb(x_2,Y^{x_2,y_1}_t)\cdot \nabla_x Y^{x_2,y_1}_t\right]\right.\\
&&\quad\quad\quad\left.-\EE\left[\nabla_yb(x_1,Y^{x_1,y_2}_t)\cdot \nabla_x Y^{x_1,y_2}_t-\nabla_yb(x_2,Y^{x_2,y_2}_t)\cdot \nabla_x Y^{x_2,y_2}_t\right]\right\|.
\end{eqnarray*}
By the assumption that $\nabla_x\nabla_yb(x,y)$ is H\"{o}lder continuous with respect to $x$
with index $\gamma$ and Lemma \ref{L3.6}, we get that
\begin{eqnarray}
Q_{11}\leq\!\!\!\!\!\!\!\!&&\EE\left\|\int^1_0 \left[ \nabla_x\nabla_yb(x_1, \xi Y^{x_1,y_1}_t+(1-\xi)Y^{x_1,y_2}_t)\right]d\xi\cdot (Y^{x_1,y_1}_t-Y^{x_1,y_2}_t)\right.\nonumber\\
&&\left.-\int^1_0 \left[\nabla_x\nabla_yb(x_2, \xi Y^{x_1,y_1}_t+(1-\xi)Y^{x_1,y_2}_t)\right]d\xi\cdot ( Y^{x_1,y_1}_t-Y^{x_1,y_2}_t)\right\|\nonumber\\
\leq\!\!\!\!\!\!\!\!&&  C|x_1-x_2|^{\gamma}\EE |Y^{x_1,y_1}_t-Y^{x_1,y_2}_t|
\leq  C|x_1-x_2|^{\gamma}|y_1-y_2|e^{-\frac{\beta t}{2}}.\label{Q11}
\end{eqnarray}
By the boundedness of $\|\nabla_x\partial^2_{y}b(x,y)\|$ and $\|\nabla_x\nabla_yb(x,y)\|$, we obtain
\begin{eqnarray}
Q_{12}\leq\!\!\!\!\!\!\!\!&&\EE\left\|\int^1_0\left[\nabla_x\nabla_yb(x_2,\xi Y^{x_1,y_1}_t+(1-\xi)Y^{x_2,y_1}_t)\right]d\xi \cdot (Y^{x_1,y_1}_t-Y^{x_2,y_1}_t)\right.\nonumber\\
\!\!\!\!\!\!\!\!&&\quad\left.-\int^1_0\left[\nabla_x\nabla_yb(x_2,\xi Y^{x_1,y_2}_t+(1-\xi)Y^{x_2,y_2}_t)\right]d \xi \cdot (Y^{x_1,y_2}_t-Y^{x_2,y_2}_t)\right\|\nonumber\\
\leq\!\!\!\!\!\!\!\!&&\EE \left[\int^1_0\left|\xi(Y^{x_1,y_1}_t-Y^{x_1,y_2}_t)+(1-\xi)(Y^{x_2,y_1}_t-Y^{x_2,y_2}_t)\right|d\xi |Y^{x_1,y_1}_t-Y^{x_2,y_1}_t|\right]\nonumber\\
\!\!\!\!\!\!\!\!&&\quad+\EE|Y^{x_1,y_1}_t-Y^{x_2,y_1}_t-Y^{x_1,y_2}_t+Y^{x_2,y_2}_t|\nonumber\\
\leq\!\!\!\!\!\!\!\!&&\EE\left[\left(|Y^{x_1,y_1}_t-Y^{x_1,y_2}_t|+|Y^{x_2,y_1}_t-Y^{x_2,y_2}_t|\right)|Y^{x_1,y_1}_t-Y^{x_2,y_1}_t|\right]\nonumber\\
\!\!\!\!\!\!\!\!&&+\EE\left|\int^1_0 \nabla_xY^{\xi x_1+(1-\xi)x_2,y_1}_t d\xi\cdot (x_1-x_2)-\int^1_0 \nabla_xY^{\xi x_1+(1-\xi)x_2,y_2}_t d\xi\cdot (x_1-x_2)\right|\nonumber\\
\leq\!\!\!\!\!\!\!\!&&C|x_1-x_2||y_1-y_2|e^{-\frac{\beta t}{2}},\label{Q12}
\end{eqnarray}
where the last inequality is due to Lemma \ref{L3.6} and estimate \eref{partial xy Y}.
Note that
\begin{align*}
Q_{13}&\leq\EE\big\|\left[\nabla_yb(x_1,Y^{x_1,y_1}_t)-\nabla_yb(x_2,Y^{x_1,y_1}_t)\right]\cdot \nabla_x Y^{x_1,y_1}_t\\
&\qquad-\left[\nabla_yb(x_1,Y^{x_1,y_2}_t)-\nabla_yb(x_2,Y^{x_1,y_2}_t)\right]\cdot \nabla_x Y^{x_1,y_2}_t\big\|\\
&\quad+\EE\big\|\left[\nabla_yb(x_2,Y^{x_1,y_1}_t)-\nabla_yb(x_2,Y^{x_2,y_1}_t)\right]\cdot \nabla_x Y^{x_1,y_1}_t\\
&\qquad-\left[\nabla_yb(x_2,Y^{x_1,y_2}_t)-\nabla_yb(x_2,Y^{x_2,y_2}_t)\right]\cdot \nabla_x Y^{x_1,y_2}_t\big\|\\
&\quad+\EE \big\|\nabla_yb(x_2,Y^{x_2,y_1}_t)\cdot (\nabla_x Y^{x_1,y_1}_t-\nabla_x Y^{x_2,y_1}_t)\\
&\qquad-\nabla_yb(x_2,Y^{x_2,y_2}_t)\cdot (\nabla_x Y^{x_1,y_2}_t-\nabla_x Y^{x_2,y_2}_t)\big\|=:Q_{131}+Q_{132}+Q_{133}.
\end{align*}
Since $\|\nabla_x\nabla_yb(x,y)\|$, $\|\nabla^2_{y}b(x,y)\|$, $\|\nabla_x\nabla^2_{y}b(x,y)\|$ are bounded and $\nabla^2_{y}b(x,y)$ is H\"{o}lder continuous with respect to $y$ with index $\delta$, we have
\begin{eqnarray}
Q_{131}\leq \!\!\!\!\!\!\!\!&& \EE\left\| \int^1_0\nabla_x\nabla_y b(\xi x_1+(1-\xi)x_2,Y^{x_1,y_1}_t)d\xi\cdot \left(x_1-x_2,\nabla_x Y^{x_1,y_1}_t\right)\right.\nonumber\\
&&\quad\left.-\int^1_0\nabla_x\nabla_y b(\xi x_1+(1-\xi)x_2,Y^{x_1,y_2}_t)d\xi\cdot \left(x_1-x_2, \nabla_x Y^{x_1,y_2}_t\right)\right\|\nonumber\\
\leq \!\!\!\!\!\!\!\!&&C|x_1\!-\!x_2|\EE\left(|Y^{x_1,y_1}_t\!-\!Y^{x_1,y_2}_t|\|\nabla_x Y^{x_1,y_1}_t\|\right)\!+\!C|x_1\!-\!x_2|\EE\|\nabla_x Y^{x_1,y_1}_t\!-\!\nabla_x Y^{x_1,y_2}_t\|\nonumber\\
\leq \!\!\!\!\!\!\!\!&&Ce^{-\frac{\beta t}{2}}|x_1-x_2||y_1-y_2|\label{Q131}
\end{eqnarray}
and
\begin{eqnarray}
Q_{132}\leq\!\!\!\!\!\!\!\!&&\EE \left\|\int^1_0\nabla^2_{y}b(x_2,\xi Y^{x_1,y_1}_t+(1-\xi)Y^{x_2,y_1}_t)d\xi\cdot\left(Y^{x_1,y_1}_t-Y^{x_2,y_1}_t, \nabla_x Y^{x_1,y_1}_t\right)\right.\nonumber\\
&&\quad\left.-\int^1_0\nabla^2_{y}b(x_2,\xi Y^{x_1,y_2}_t+(1-\xi)Y^{x_2,y_2}_t)d\xi\cdot \left(Y^{x_1,y_2}_t-Y^{x_2,y_2}_t, \nabla_x Y^{x_1,y_2}_t\right)\right\|\nonumber\\
\leq \!\!\!\!\!\!\!\!&& C\EE\left[\left(|Y^{x_1,y_1}_t-Y^{x_1,y_2}_t|^{\delta}+|Y^{x_2,y_1}_t-Y^{x_2,y_2}_t|^{\delta}\right)|Y^{x_1,y_1}_t-Y^{x_2,y_1}_t|\|\nabla_x Y^{x_1,y_1}_t\|\right]\nonumber\\
&&+C\EE\left(|Y^{x_1,y_1}_t-Y^{x_2,y_1}_t-Y^{x_1,y_2}_t+Y^{x_2,y_2}_t|\|\nabla_x Y^{x_1,y_1}_t\|\right)\nonumber\\
&&+C\EE\left(|Y^{x_1,y_1}_t-Y^{x_2,y_1}_t|\|\nabla_x Y^{x_1,y_1}_t-\nabla_x Y^{x_1,y_2}_t\|\right)\nonumber\\
\leq \!\!\!\!\!\!\!\!&&Ce^{-\frac{\beta \delta t}{2}}|x_1-x_2||y_1-y_2|^{\delta}(1+|y_1-y_2|^{1-\delta}).\label{Q132}
\end{eqnarray}
The assumption $f\in C^{1+\gamma,2+\gamma}_b$ implies that $\nabla_x\nabla_yf(x,y)$ is H\"{o}lder continuous with respect to $x$ with index $\gamma$ and $\nabla^2_{y}f(x,y)$ is H\"{o}lder continuous with respect to $y$ with index $\gamma$, and $\|\nabla_x\nabla^2_{y}f(x,y)\|$ is uniformly
bounded. By a straightforward computation, we get
\begin{eqnarray}
\sup_{ y\in\RR^{d_2}} \EE\|\nabla_y\nabla_x Y^{x_1,y}_t-\nabla_y\nabla_x Y^{x_2,y}_t\|\leq C e^{-\frac{\beta t}{4}}|x_1-x_2|^{\gamma}(1+|x_1-x_2|^{1-\gamma}),\label{partial yxx}
\end{eqnarray}
and $\nabla_y\nabla_x Y^{x,y}_t$ satisfies
 \begin{align*}
d\nabla_y\nabla_x Y^{x,y}_t=&\nabla_y\nabla_x f(x,Y^{x,y}_{t})\cdot \nabla_y Y^{x,y}_tdt+\partial^2_y f(x,Y^{x,y}_{t})\cdot(\nabla_y Y^{x,y}_t,\nabla_x Y^{x,y}_t)dt\\
&+\nabla_y f(x,Y^{x,y}_{t})\cdot \nabla_y\nabla_x Y^{x,y}_t dt,
\nabla_y\nabla_x Y^{x,y}_0=0.
\end{align*}
Using \eref{partial xy Y} and \eref{partial yxx}, we get
\begin{eqnarray}
Q_{133}\leq \!\!\!\!\!\!\!\!&& \EE\left\|\left[\nabla_yb(x_2,Y^{x_2,y_1}_t)-\nabla_yb(x_2,Y^{x_2,y_2}_t)\right]\cdot(\nabla_x Y^{x_1,y_1}_t-\nabla_x Y^{x_2,y_1}_t)\right\|\nonumber\\
&&+\EE\left\|\nabla_yb(x_2,Y^{x_2,y_2}_t)\cdot (\nabla_x Y^{x_1,y_1}_t-\nabla_x Y^{x_2,y_1}_t-\nabla_x Y^{x_1,y_2}_t+\nabla_x Y^{x_2,y_2}_t)\right\|\nonumber\\
\leq \!\!\!\!\!\!\!\!&&Ce^{-\frac{\beta t}{4}}|x_1-x_2|^{\gamma}(1+|x_1-x_2|^{1-\gamma})|y_1-y_2|.\label{Q133}
\end{eqnarray}
Combining \eref{Q131}, \eref{Q132} and \eref{Q133}, we get
\begin{eqnarray}
Q_{13}\leq Ce^{-\frac{(\beta\wedge2\delta)t}{4}}|x_1-x_2|^{\gamma}|y_1-y_2|^{\delta}(1+|y_1-y_2|^{1-\delta})(1+|x_1-x_2|^{1-\gamma}).\label{Q13}
\end{eqnarray}
Finally, \eref{Q11},\eref{Q12} and \eref{Q13} together imply
\begin{eqnarray}
Q_1\leq \!\!\!\!\!\!\!\!&&Ce^{-\frac{(\beta\wedge2\delta)t}{4}}|x_1-x_2|^{\gamma} (1+|x_1-x_2|^{1-\gamma})\EE(|y-Y^{x_1,y}_{t_0}|+|y-Y^{x_1,y}_{t_0}|^{\delta})\nonumber\\
\leq \!\!\!\!\!\!\!\!&&Ce^{-\frac{(\beta\wedge2\delta)t}{4}}|x_1-x_2|^{\gamma}(1+|x_1-x_2|^{1-\gamma})(1+|y|).\label{Q1}
\end{eqnarray}

(ii) For the term $Q_2$, note that
\begin{eqnarray*}
\nabla_y\nabla_x\hat b(x, y, t)=\!\!\!\!\!\!\!\!&& \nabla_y\EE \left[\nabla_x b(x, Y^{x,y}_{t})\right]+\nabla_y\EE\left[ \nabla_yb(x,Y^{x,y}_{t})\cdot\nabla_xY^{x,y}_t \right]\nonumber\\
=\!\!\!\!\!\!\!\!&& \EE \left[ \nabla_x\nabla_yb(x,Y^{x,y}_t)\cdot\nabla_yY^{x,y}_t\right]+\EE\left[ \nabla^2_{y}b(x,Y^{x,y}_{t})\cdot(\nabla_xY^{x,y}_t, \nabla_y Y^{x,y}_t)\right]\\
&&+ \EE \left[ \nabla_yb(x,Y^{x,y}_t)\cdot\nabla_x\nabla_yY^{x,y}_t\right].
\end{eqnarray*}
\eref{partial xy Y} implies
$$
\sup_{x\in\RR^{d_1}, y\in\RR^{d_2}}\|\nabla_x\nabla_yY^{x,y}_t\|\leq Ce^{-\frac{\beta t}{4}}.
$$
Combining this with with \eref{S0} and \eref{partial yY}, we get
\begin{eqnarray}
\sup_{x\in\RR^{d_1}, y\in\RR^{d_2}}\left\|\nabla_x\nabla_y\hat b( x,y, t)\right\|\leq C e^{-\frac{\beta t}{4}}.\label{partial xy hatb}
\end{eqnarray}
Thus
\begin{eqnarray}
Q_2\leq\!\!\!\!\!\!\!\!&& C e^{-\frac{\beta t}{4}}\EE|Y^{x_2,y}_{t_0}-Y^{x_1,y}_{t_0}|\leq C e^{-\frac{\beta t}{4}}|x_1-x_2|.\label{Q2}
\end{eqnarray}

(iii) For the term $Q_3$, by a similar argument as in the proof of \eref{partial xy hatb}, we have
$$
\sup_{x\in\RR^{d_1}, y\in\RR^{d_2}}\left\|\nabla^2_{y}\hat b(x, y, t)\right\|\leq C e^{-\frac{\beta t}{4}},
$$
which, together with \eref{partial xy hatb}, implies
\begin{eqnarray}
Q_3\leq C e^{-\frac{\beta t}{4}}|x_1-x_2|.\label{Q3}
\end{eqnarray}
Combining \eref{Q1}, \eref{Q2} and \eref{Q3}, we get \eref{E22}.
The proof is complete.
\end{proof}
\section{Proof of main results}
In this section, we give the proofs of Theorem \ref{main result 1} and Theorem \ref{main result 2}. Our arguments is based the Poisson equation and is inspired by \cite{B2} (see also \cite{PV1,PV2,RSX,RSX2}).

\subsection{The Proof of Theorem \ref{main result 1}}
Note that
\begin{eqnarray*}
X_{t}^{\ep}-\bar{X}_{t}=\!\!\!\!\!\!\!\!&&\int_{0}^{t}\left[b(X_{s}^{\ep},Y_{s}^{\ep})-\bar{b}(\bar{X}_{s})\right]ds\\
=\!\!\!\!\!\!\!\!&&\int_{0}^{t}\left[b(X_{s}^{\ep},Y_{s}^{\ep})-\bar{b}(X^{\ep}_{s})\right]ds+\int_{0}^{t}\left[\bar{b}(X^{\ep}_{s})-\bar{b}(\bar{X}_{s})\right]ds.
\end{eqnarray*}
Using  the Lipschitz continuity of $\bar{b}$, one can easily show that for any $p\in [1,\alpha)$,
\begin{eqnarray*}
&&\EE\left(\sup_{t\in [0, T]}|X_{t}^{\ep}-\bar{X}_{t}|^p\right)\\
\leq\!\!\!\!\!\!\!\!&& C_p\EE\left[\sup_{t\in[0,T]}\left|\int_{0}^{t}b(X_{s}^{\ep},Y_{s}^{\ep})-\bar{b}(X^{\ep}_{s})ds\right|^p\right]+C_{p,T}\EE\int_{0}^{T}|X_{t}^{\ep}-\bar{X}_{t}|^p dt.
\end{eqnarray*}
Grownall's inequality implies
\begin{eqnarray}
\EE\left(\sup_{t\in [0, T]}|X_{t}^{\ep}-\bar{X}_{t}|^p\right)\leq\!\!\!\!\!\!\!\!&&C_{p,T}\EE\left[\sup_{t\in[0,T]}\left|\int_{0}^{t}b(X_{s}^{\ep},Y_{s}^{\ep})-\bar{b}(X^{\ep}_{s})ds\right|^p\right].\label{I3.10}
\end{eqnarray}
By Proposition \ref{P3.6}, there exists a function $\Phi(x,y)$ such that
$\Phi(\cdot,y)\in C^{1}(\RR^{d_1},\RR^{d_1})$, $\Phi(x,\cdot)\in C^{2}(\RR^{d_2}, \RR^{d_1})$ and
\begin{eqnarray}
-\mathscr{L}_{2}(x,y)\Phi(x,y)=b(x,y)-\bar{b}(x).\label{PE}
\end{eqnarray}
Moreover, the estimates \eref{E1} and \eref{E2} hold. By It\^o's formula (see \cite[Theorem 4.4.7]{A}), we have
\begin{eqnarray*}
\Phi(X_{t}^{\ep},Y^{\ep}_{t})=\!\!\!\!\!\!\!\!&&\Phi(x,y)+\int^t_0 \mathscr{L}_{1}(Y^{\ep}_{r})\Phi(X_{r}^{\ep},Y^{\ep}_{r})dr\\
&&+\frac{1}{\ep}\int^t_0 \mathscr{L}_{2}(X_{r}^{\ep},Y^{\ep}_{r})\Phi(X_{r}^{\ep},Y^{\ep}_{r})dr+M^{\ep,1}_{t}+M^{\ep,2}_{t},
\end{eqnarray*}
where $\mathscr{L}_{2}(x,y)\Phi(x,y)$ is defined by \eref{L_2} and
\begin{eqnarray*}
\mathscr{L}_{1}(y)\Phi(x,y):=-(-\Delta_{x})^{\alpha/2}\Phi(x,y)+\langle b(x,y), \nabla_x \Phi(x,y)\rangle,
\end{eqnarray*}
and $M^{\ep,1}_{t}, M^{\ep,2}_{t}$ are two $\mathscr{F}_{t}$-martingales defined by
$$M^{\ep,1}_{t}:=\int^t_0 \int_{\RR^{d_1}}\Phi(X_{r-}^{\ep}+x, Y^{\ep}_{r-})-\Phi(X_{r-}^{\ep}, Y^{\ep}_{r-})\tilde{N}^1(dr,dx),$$
$$M^{\ep,2}_{t}:=\int^t_0 \int_{\RR^{d_2}}\Phi(X_{r-}^{\ep}, Y^{\ep}_{r-}+\ep^{-1/\alpha}y)-\Phi(X_{r-}^{\ep}, Y^{\ep}_{r-})\tilde{N}^2(dr,dy)$$
and $\tilde{N}^i$ $(i=1,2)$ are defined by (\ref{tn}). Consequently, we have
\begin{eqnarray}
\int^t_0 -\mathscr{L}_{2}(X_{r}^{\ep},Y^{\ep}_{r})\Phi(X_{r}^{\ep},Y^{\ep}_{r})dr=\!\!\!\!\!\!\!\!&&\ep\Big[\Phi(x,y)-\Phi(X_{t}^{\ep},Y^{\ep}_{t})\nonumber\\
&&+\int^t_0 \mathscr{L}_{1}(Y^{\ep}_{r})\Phi(X_{r}^{\ep},Y^{\ep}_{r})dr+M^{\ep,1}_{t}+M^{\ep,2}_{t}\Big].\label{PT}
\end{eqnarray}
Combining \eref{I3.10}, \eref{PE} and \eref{PT}, we get
\begin{eqnarray}
&&\EE\left(\sup_{t\in [0, T]}|X_{t}^{\ep}-\bar{X}_{t}|^p\right)\leq C_{p,T}\EE\left[\sup_{t\in[0,T]}\left|\int_{0}^{t}-\mathscr{L}_{2}(X_{r}^{\ep},Y^{\ep}_{r})\Phi(X_{r}^{\ep},Y^{\ep}_{r})dr\right|^p\right]\nonumber\\
\leq\!\!\!\!\!\!\!\!&&C_{p,T}\ep^{p}\Bigg[\EE\left(\sup_{t\in[0,T]}|\Phi(x,y)-\Phi(X_{t}^{\ep},Y^{\ep}_{t})|^p\right)+\EE\int^T_0 \left|\mathscr{L}_{1}(Y^{\ep}_{r})\Phi(X_{r}^{\ep},Y^{\ep}_{r})\right|^p dr\nonumber\\
&&+\EE\left(\sup_{t\in[0,T]}|M^{\ep,1}_{t}|^p\right)+\EE\left(\sup_{t\in[0,T]}|M^{\ep,2}_{t}|^p\right)\Bigg].\label{F5.4}
\end{eqnarray}
By \eref{E1} and \eref{Yvare}, we have
\begin{eqnarray}
\EE\left(\sup_{t\in[0,T]}|\Phi(x,y)-\Phi(X_{t}^{\ep},Y^{\ep}_{t})|^p\right)\leq\!\!\!\!\!\!\!\!&&C(1+|y|^p)+\EE\left(\sup_{t\in[0,T]}|Y^{\ep}_{t}|^p\right)\nonumber\\
\leq\!\!\!\!\!\!\!\!&&C_{p,T}(1+|y|^p)\vare^{-p/\alpha}.\label{F5.5}
\end{eqnarray}
It follows from  \eref{E2} and \eref{E3} that
\begin{eqnarray}
&&\EE\int^T_0\left|\mathscr{L}_{1}(Y^{\ep}_{r})\Phi(X^{\ep}_{r},Y^{\ep}_{r})\right|^p dr\nonumber\\
\leq\!\!\!\!\!\!\!\!&&C_p\EE\int^T_0\left[\int_{|z|\leq 1}|\Phi(X^{\ep}_{s}+z,Y^{\ep}_{s})-\Phi(X^{\ep}_{s},Y^{\ep}_{s})-\langle z, \nabla_x \Phi(X^{\ep}_{s},Y^{\ep}_{s})\rangle |\nu_1(dz)\right]^pds\nonumber\\
&&+C_{p,T}\EE\int^T_0\left[\int_{|z|> 1}|\Phi(X^{\ep}_{s}+z,Y^{\ep}_{s})-\Phi(X^{\ep}_{s},Y^{\ep}_{s})|\nu_1(dz)\right]^p ds\nonumber\\
&&+C_{p,T}\EE\int^T_0|\langle b(X^{\ep}_{s},Y^{\ep}_{s}), \nabla_x \Phi(X^{\ep}_{s},Y^{\ep}_{s})\rangle|^p ds\nonumber\\
\leq\!\!\!\!\!\!\!\!&&C_{p,T}\EE\int^T_0\left[\int_{|z|\leq 1}|z|^{\gamma+1}(1+|z|^{1-\gamma})\nu_1(dz)\right]^p(1+|Y^{\ep}_{s}|^p)ds\nonumber\\
&&+C_{p,T}\EE\int^T_0\left[\int_{|z|> 1}|z|\nu_1(dz)\right]^p (1+|Y^{\ep}_{s}|^p)ds\nonumber\\
&&+C_{p,T}\EE\int^T_0 \left(1+|X^{\ep}_{s}|^p+|Y^{\ep}_{s}|^p\right)(1+|Y^{\ep}_{s}|^{\theta})ds\nonumber\\
\leq\!\!\!\!\!\!\!\!&&C_{p,T}(1+|x|^p+|y|^p)+C_{p,T}\EE\int^T_0 \left(1+|X^{\ep}_{s}|^{p'}+|Y^{\ep}_{s}|^{\frac{\theta p'}{p'-p}\vee (p+\theta)}\right)ds\nonumber\\
\leq\!\!\!\!\!\!\!\!&&C_{p,T}\left((1+|x|^{p'}+|y|^{\frac{\theta p'}{p'-p}\vee (p+\theta)}\right),\label{F5.6}
\end{eqnarray}
where $p<p'<\alpha$ and $\theta$ is small enough such that $\frac{\theta p'}{p'-p}\vee (p+\theta)<\alpha$.

Using Burkholder-Davis-Gundy's inequality  (see e.g. \cite[Lemma 8.22]{PZ1}) and \eref{E2}, we get for any $\theta\in (0,1/2]$,
\begin{eqnarray}
&&\EE\left(\sup_{t\in[0,T]}|M^{\ep,1}_{t}|^p\right)\nonumber\\
\leq\!\!\!\!\!\!\!\!&&C_p\EE\left[\sup_{t\in[0,T]}\left|\int^t_0\int_{|x|\leq 1}\Phi(X^{\ep}_{s-}+x,Y^{\ep}_{s-})-\Phi(X^{\ep}_{s-},Y^{\ep}_{s-})\tilde{N}^1(ds,dx)\right|^p\right]\nonumber\\
&&+C_p\EE\left[\sup_{t\in[0,T]}\left|\int^t_0\int_{|x|>1}\Phi(X^{\ep}_{s-}+x,Y^{\ep}_{s-})-\Phi(X^{\ep}_{s-},Y^{\ep}_{s-})\tilde{N}^1(ds,dx)\right|^p\right]\nonumber\\
\leq\!\!\!\!\!\!\!\!&&C_p\EE\left|\int^T_0\int_{|x|\leq 1}\left|\Phi(X^{\ep}_{s-}+x,Y^{\ep}_{s-})-\Phi(X^{\ep}_{s-},Y^{\ep}_{s-})\right|^2 N^1(ds,dx)\right|^{p/2}\nonumber\\
&&+C_p\EE\int^T_0\int_{|x|>1}\left|\Phi(X^{\ep}_{s-}+x,Y^{\ep}_{s-})-\Phi(X^{\ep}_{s-},Y^{\ep}_{s-})\right|^p\nu_1(dx)ds\nonumber\\
\leq\!\!\!\!\!\!\!\!&&C_p\left[\EE\int^T_0\int_{|x|\leq 1}|x|^2 \nu_1(dx)(1+|Y^{\ep}_{s}|^{2\theta})ds\right]^{p/2}\nonumber\\
&&+C_p\EE\int^T_0\int_{|x|>1}|x|^{p}\nu_1(dx)(1+|Y^{\ep}_{s}|^{p\theta})ds
\leq C_{p,T}(1+|y|^p),\label{F5.7}
\end{eqnarray}
and by  \eref{E1}, we have
\begin{eqnarray}
&&\EE\left(\sup_{t\in[0,T]}|M^{\ep,2}_{t}|^p\right)\nonumber\\
\leq\!\!\!\!\!\!\!\!&& C_p\EE\left[\sup_{t\in[0,T]}\left|\int^t_0\int_{|y|\leq 1}\Phi(X^{\ep}_{s-},Y^{\ep}_{s-}+\ep^{-1/\alpha}y)-\Phi(X^{\ep}_{s-},Y^{\ep}_{s-})\tilde{N}^2(ds,dy)\right|^p\right]\nonumber\\
&&+C_p\EE\left[\sup_{t\in[0,T]}\left|\int^t_0\int_{|y|>1}\Phi(X^{\ep}_{s-},Y^{\ep}_{s-}+\ep^{-1/\alpha}y)-\Phi(X^{\ep}_{s-},Y^{\ep}_{s-})\tilde{N}^2(ds,dy)\right|^p\right]\nonumber\\
\leq\!\!\!\!\!\!\!\!&&C_p\EE\left|\int^T_0\int_{|y|\leq 1}|\Phi(X^{\ep}_{s-},Y^{\ep}_{s-}+\ep^{-1/\alpha}y)-\Phi(X^{\ep}_{s-},Y^{\ep}_{s-})|^2 N^2(ds,dy)\right|^{p/2}\nonumber\\
&&+C_p\EE\int^T_0\int_{|y|>1}|\Phi(X^{\ep}_{s-},Y^{\ep}_{s-}+\ep^{-1/\alpha}y)-\Phi(X^{\ep}_{s-},Y^{\ep}_{s-})|^{p}\nu_2(dy)ds\nonumber\\
\leq\!\!\!\!\!\!\!\!&&C_p\ep^{-p/\alpha}\left\{\left[\int^T_0\int_{|y|\leq 1}|y|^2 \nu_2(dy)ds\right]^{p/2}\!\!+\!\int^T_0\int_{|y|>1}|y|^{p}\nu_2(dy)ds\right\}\nonumber\\
\leq\!\!\!\!\!\!\!\!&&C_{p,T}\ep^{-p/\alpha},\label{F5.8}
\end{eqnarray}
where we have also used the fact that $\Phi(x,\cdot)\in C^1_b(\RR^{d_2})$. Hence, \eref{F5.4}-\eref{F5.8} imply that
\begin{eqnarray*}
\mathbb{E}\left(\sup_{t\in[0,T]}|X_{t}^{\ep}-\bar{X}_{t}|^p\right)\leq C_{p,T}(1+|x|^p+|y|^p)\ep^{p(1-1/\alpha)}.
\end{eqnarray*}
The proof is complete.

\vskip 0.2cm
\subsection{The Proof of Theorem \ref{main result 2}}

We consider the following Kolmogorov equation:
\begin{equation}\left\{\begin{array}{l}\label{KE}
\displaystyle
\partial_t u(t,x)=\bar{\mathscr{L}}_1 u(t,x),\quad t\in[0, T], \\
u(0, x)=\phi(x),
\end{array}\right.
\end{equation}
where $\phi\in C^{2+\gamma}_b(\RR^{d_1})$ and $\bar{\mathscr{L}}_1$ is the infinitesimal generator of the transition semigroup of the averaged equation \eref{1.3}, which is given by
\begin{eqnarray*}
\bar{\mathscr{L}}_1\phi(x):=-(-\Delta_x)^{\alpha/2}\phi(x)+\langle \bar{b}(x), \nabla_x  \phi(x)\rangle.
\end{eqnarray*}
Since $b,f\in C^{2+\gamma,2+\gamma}_b$, one can check by straightforward computation   that $\bar{b}\in C^{2+\gamma}_b(\RR^{d_1})$. Thus, equation \eref{KE} has a unique solution $u$ which is given by
$$
u(t,x)=\EE\phi(\bar{X}_t(x)),\quad t\in [0,T].
$$
Furthermore, $u(t,\cdot)\in C^{2+\gamma}_b(\RR^{d_1}), \nabla_x u(\cdot,x)\in C^1([0,T])$ and there exists $C_T>0$ such that
\begin{eqnarray}
\sup_{t\in[0, T]}\|u(t,\cdot)\|_{C^{2+\gamma}_b}\leq C_T,\quad \sup_{t\in[0, T],x\in\RR^{d_1}}\|\partial_t(\nabla_x u(t,x))\|\leq C_T.\label{UE}
\end{eqnarray}

Now we are in a position to give:

\begin{proof}[Proof of Theorem \ref{main result 2}]
For fixed $t>0$, let $\tilde{u}^t(s,x):=u(t-s,x)$, $s\in [0,t]$. By It\^{o}'s formula, we have
\begin{eqnarray*}
\tilde{u}^t(t, X^{\ep}_t)=\!\!\!\!\!\!\!\!&&\tilde{u}^t(0,x)+\int^t_0 \partial_s \tilde{u}^t(s, X^{\ep}_s )ds+\int^t_0 \mathscr{L}_{1}(Y^{\ep}_s)\tilde{u}^t(s, X^{\ep}_s)ds+\tilde{M}_t,
\end{eqnarray*}
where $\tilde{M}_t$ is a $\mathscr{F}_{t}$-martingale defined by,
$$
\tilde{M}_t:=\int^t_0 \int_{\RR^{d_1}}\tilde{u}^t(s,X_{s-}^{\ep}+x)-\tilde{u}^t(s,X_{s-}^{\ep})\tilde{N}^1(dx,ds).
$$
Note that $\tilde{u}^t(t, X^{\ep}_t)=\phi(X^{\ep}_t)$, $\tilde{u}^t(0, x)=\EE\phi(\bar{X}_t(x))$ and $\nabla_s \tilde{u}^t(s, X^{\ep}_s )=-\bar{\mathscr{L}}_1 \tilde{u}^t(s, X^{\ep}_s)$, we have
\begin{eqnarray}
\left|\EE\phi(X^{\ep}_{t})-\EE\phi(\bar{X}_{t})\right|=\!\!\!\!\!\!\!\!&&\left|\EE\int^t_0 -\bar{\mathscr{L}}_1 \tilde{u}^t(s, X^{\ep}_s )ds+\EE\int^t_0 \mathscr{L}_{1}(Y^{\ep}_s)\tilde{u}^t(s, X^{\ep}_s)ds\right|\nonumber\\
=\!\!\!\!\!\!\!\!&&\left|\EE\int^t_0 \langle b(X^{\ep}_s,Y^{\ep}_s)-\bar{b}(X^{\ep}_s), \nabla_x \tilde{u}^t(s, X^{\ep}_s )\rangle ds \right|.\label{F5.11}
\end{eqnarray}
For any $s\in [0,t], x\in\RR^{d_1},y\in\RR^{d_2}$, define
$$
F^t(s,x,y):=\langle b(x,y), \nabla_x \tilde{u}^t(s, x)\rangle
$$
and $
\bar{F}^t(s,x):=\int_{\RR^{d_2}} F^t(s,x,y)\mu^x(dy)=\langle \bar{b}(x), \nabla_x \tilde{u}^t(s, x)\rangle.
$
Since $b$ is bounded, $b\in C^{2+\gamma,2+\gamma}_b$ and $\tilde{u}^t(s,\cdot)\in C^{2+\gamma}_b(\RR^{d_1})$, we have
$$
F^t(s,\cdot,\cdot)\in C^{1+\gamma,2+\gamma}_b,\quad \nabla_sF^t(s,x,\cdot)\in C^1_b(\RR^{d_2}).
$$
Using \eref{UE} and an argument similar to that used in the proof of Proposition \ref{P3.6}, we can get that
$$
\tilde{\Phi}^t(s, x,y):=\int^{\infty}_0  \left[\EE F^t(s,x,Y^{x,y}_r)-\bar{F}^t(s,x)\right]dr
$$
is a solution of the following Poisson equation:
\begin{eqnarray}
-\mathscr{L}_{2}(x,y)\tilde{\Phi}^t(s,x,y)=F^t(s,x,y)-\bar{F}^t(s,x),\quad s\in [0,t].\label{WPE}
\end{eqnarray}
Moreover, $\tilde{\Phi}^t(\cdot, x,y)\in C^{1}([0, t])$, $\tilde{\Phi}^t(s, \cdot,y)\in C^{1}(\RR^{d_1})$, $\tilde{\Phi}^t(s, x,\cdot)\in C^{2}(\RR^{d_2})$ and for any $T>0$, $t\in [0,T]$,$\theta\in (0,1]$, there exist $C_T, C_{T,\theta}>0$ such that the following estimates hold:
\begin{eqnarray}
\sup_{s\in [0, t],x\in\RR^{d_1}}\left[|\tilde{\Phi}^t(s,x,y)|+|\nabla_s \tilde{\Phi}^t(s,x,y)|\right]\leq C_T(1+|y|),\label{E121}
\end{eqnarray}
\begin{eqnarray}
\sup_{s\in [0, t],x\in\RR^{d_1}}\left|\nabla_x \tilde{\Phi}^t(s,x,y)\right|\leq C_{T,\theta}(1+|y|^{\theta}),\label{E122}
\end{eqnarray}
\begin{eqnarray}
\sup_{s\in[0,t]}\left\| \nabla_x\tilde{\Phi}^t(s,x_1, y)\!- \!\nabla_x\tilde{\Phi}^t(s,x_2, y)\right\|\leq \!\!C_T|x_1\!-\!x_2|^{\gamma}\!(1\!+\!|x_1\!-\!x_2|^{1-\gamma})(1\!+\!|y|). \label{E221}
\end{eqnarray}
Using It\^o's formula and taking expectation on both sides, we get
\begin{eqnarray*}
\EE\tilde{\Phi}^t(t, X_{t}^{\ep},Y^{\ep}_{t})=\!\!\!\!\!\!\!\!&&\tilde \Phi^t(0, x,y)\!+\!\EE\int^t_0 \nabla_s \tilde{\Phi}^t(s, X_{s}^{\ep},Y^{\ep}_{s})ds\!+\!\EE\int^t_0\mathscr{L}_{1}(Y^{\ep}_{s})\tilde\Phi^t(s, X_{s}^{\ep},Y^{\ep}_{s})ds\\
&&+\frac{1}{\ep}\EE\int^t_0 \mathscr{L}_{2}(X_{s}^{\ep},Y^{\ep}_{s})\tilde{\Phi}^t(s, X_{s}^{\ep},Y^{\ep}_{s})ds,
\end{eqnarray*}
which implies
\begin{align}
&-\EE\int^t_0 \mathscr{L}_{2}(X_{s}^{\ep},Y^{\ep}_{s})\tilde\Phi^t(s, X_{s}^{\ep},Y^{\ep}_{s})ds= \ep\big[\tilde{\Phi}^t(0, x,y)-\EE\tilde{\Phi}^t(t, X_{t}^{\ep},Y^{\ep}_{t})\nonumber\\
&\qquad+\EE\int^t_0 \nabla_s \tilde{\Phi}^t(s, X_{s}^{\ep},Y^{\ep}_{s})ds +\EE\int^t_0\mathscr{L}_{1}(Y^{\ep}_{s})\tilde{\Phi}^t(s, X_{s}^{\ep},Y^{\ep}_{s})ds\big].\label{F3.39}
\end{align}
Combining  \eref{F5.11}, \eref{WPE} and \eref{F3.39}, we get
\begin{eqnarray*}
\sup_{t\in[0,T]}\left|\EE\phi(X^{\ep}_{t})-\EE\phi(\bar{X}_{t})\right|=\!\!\!\!\!\!\!\!&&\sup_{t\in[0,T]}\left|\EE\int^t_0 \mathscr{L}_{2}(X_{s}^{\ep},Y^{\ep}_{s})\tilde{\Phi}^t(s, X_{s}^{\ep},Y^{\ep}_{s})ds\right|\\
\leq\!\!\!\!\!\!\!\!&&\ep\Bigg[\sup_{t\in [0,T]}|\tilde{\Phi}^t(0, x,y)|+\sup_{t\in[0,T]}\left|\EE\tilde{\Phi}^t(t, X_{t}^{\ep},Y^{\ep}_{t})\right|\nonumber\\
&&+\sup_{t\in [0,T]}\EE\int^t_0\left|\nabla_s \tilde{\Phi}^t(s, X_{s}^{\ep},Y^{\ep}_{s})\right|ds\nonumber\\
&&+\sup_{t\in [0,T]}\EE\int^t_0\left|\mathscr{L}_{1}(Y^{\ep}_{s})\tilde{\Phi}^t(s, X_{s}^{\ep},Y^{\ep}_{s})\right|ds\Bigg].
\end{eqnarray*}
Finally, using \eref{E121}, \eref{E122}, \eref{E221} and an argument similar to that used in the proof of \eref{F5.6}, we easily get
\begin{eqnarray*}
\sup_{t\in[0,T]}\left|\EE\phi(X^{\ep}_{t})-\EE\phi(\bar{X}_{t})\right|\leq C \ep,
\end{eqnarray*}
where $C$ is a constant depending on $T$, $x$ and $y$. The proof is complete.
\end{proof}

\section{Appendix}

Let us first recall some facts about the isotropic $\alpha$-stable processes. For $k=1,2$, let $L^{k}_{t}$ be two isotropic $\alpha$-stable processes in $\RR^{d_k}$.
The associated Poisson random measure are defined by (see e.g. \cite{A})
$$
N^{k}(t,\Gamma)=\sum_{s\leq t}1_{\Gamma}(L^{k}_s-L^{k}_{s-}),\quad \forall\Gamma\in \mathcal{B}(\RR^{d_k}),
$$
where $\mathcal{B}(\RR^{d_k})$ are  the Borel $\sigma$-algebra of $\RR^{d_k}$. The corresponding compensated Poisson measure are given by
\begin{align}\label{tn}
\widetilde{N}^{k}(t,\Gamma)=N^{k}(t,\Gamma)-t\nu_k(\Gamma),
\end{align}
where  $\nu_k(dy):=\frac{c_{\alpha,d_k}}{|y|^{d_k+\alpha}}dy$ are the L\'evy measures, and $c_{\alpha,d_k}>0$ are constants.
By L\'evy-It\^o's decomposition and the symmetric of $\nu_k(dy)$, one has for any $c>0$,
\begin{eqnarray}
L^{k}_{t}=\int_{|x|\leq c}x\widetilde{N}^{k}(t,dx)+\int_{|x|>c}x N^{k}(t,dx).\label{555}
\end{eqnarray}

Concerning the multiscale system (\ref{Equation}), we have the following result.

\begin{lemma} \label{PMY}
Suppose the assumptions in Theorem \ref{main result 1} hold.
Then for any $\ep>0$, initial value $x\in\RR^{d_1}, y\in \RR^{d_2}$, there exists a unique strong solution $\{(X^{\ep}_t,Y^{\ep}_t), t\geq 0\}$ to system \eref{Equation}.
Moreover, for any $p\in[1,\alpha)$ and  $T>0$, there exist constants $ C_p$ and $C_{p,T}>0$ such that
\begin{eqnarray}
\sup_{\ep\in(0,1)}\mathbb{E}\left(\sup_{t\in [0, T]}|X_{t}^{\ep}|^p\right)\leq C_{p,T}(1+|x|^p+|y|^p)\label{X}
\end{eqnarray}
and
\begin{eqnarray}
\sup_{\ep\in(0,1)}\sup_{t\geq 0}\mathbb{E}|Y_{t}^{\ep}|^p\leq C_{p}(1+|y|^p).\label{Y}
\end{eqnarray}
\end{lemma}
\begin{proof}
Since $b,f$ are globally Lipschitz continuous with respect to $(x,y)$, by \cite[Theorems 6.2.3 and Theorem 6.2.11]{A}, there exists a unique solution $\{(X^{\ep}_t,Y^{\ep}_t), t\geq 0\}$ to the system (\ref{Equation}).
By Burkholder-Davis-Gundy's inequality (see e.g. \cite[Lemma 8.22]{PZ1}), we have for any $1\leq p< \alpha$,
\begin{align*}
&\EE\left(\sup_{t\in [0,T]}|L^{1}_t|^p\right)\leq  C_{p}\EE \left[\int_{|x|\leq1}|x|^2 N^{1}(T,dx)\right]^{p/2}\\
&\quad+C_{p,T}\int_{|x|>1}|x|^p \nu_1(dx)+C_{p,T}\left[\int_{|x|>1}|x|\nu_{1}(dx)\right]^p\\
&\leq C_{p,T}\left\{\left[\int_{|x|\leq1}|x|^2 \nu_1(dx)\right]^{p/2}+\int_{|x|>1}|x|^p \nu_1(dx)+\left[\int_{|x|>1}|x|\nu_{1}(dx)\right]^p\right\}\leq  C_{p,T}.
\end{align*}
It is easy to see
\begin{eqnarray*}
\mathbb{E}\left(\sup_{t\in [0,T]}|X^{\ep}_t|^p\right)\leq\!\!\!\!\!\!\!\!&&C_{p}|x|^p+C_{p,T}\int^{T}_{0}\EE|X^{\ep}_s|^pds+C_{p,T}\int^{T}_{0}\EE|Y^{\ep}_s|^p ds\\
&&+C_{p}\EE\left(\sup_{0\leq t\leq T}|L^{1}_t|^p\right)\nonumber\\
\leq\!\!\!\!\!\!\!\!&&C_{p,T}(|x|^p+1)+C_{p,T}\int^{T}_{0}\EE|X^{\ep}_s|^pds+C_{p,T}\int^{T}_{0}\EE|Y^{\ep}_s|^p ds.
\end{eqnarray*}
Thus, Grownall's inequality yields
\begin{eqnarray}
\mathbb{E}\left(\sup_{t\in [0,T]}|X^{\ep}_t|^p\right)\leq\!\!\!\!\!\!\!\!&&C_{p,T}(|x|^p+1)+C_{p,T}\int^{T}_{0}\EE|Y^{\ep}_s|^p ds.\label{F3.8}
\end{eqnarray}

Next we estimate $Y^{\vare}_t$. Define
$$U(y):=(|y|^{2}+1)^{p/2}, \quad y\in \RR^{d_2}.$$
Then for any $y\in\RR^{d_2}$, it is easy to check that
\begin{eqnarray}
|D U(y)|=\left|\frac{p y}{(|y|^2+1)^{1-p/2}}\right|\leq C_p|y|^{p-1}\label{DU}
\end{eqnarray}
and
\begin{eqnarray}
\|D^2 U(y)\|=\left\|\frac{p \text{I}_{d_2\times d_2}}{(|y|^2+1)^{1-p/2}}-\frac{p(p-2)y\otimes y}{(|y|^2+1)^{2-p/2}}\right\|\leq \frac{C_p}{(|y|^2+1)^{1-p/2}}\leq C_p. \label{D^2U}
\end{eqnarray}
Note that by (\ref{555}), $Y^{\ep}_t$ can be rewritten as
\begin{eqnarray*}
Y^{\ep}_t=\!\!\!\!\!\!\!\!&&y+\frac{1}{\ep}\int^t_0 f(X^{\ep}_{s},Y^{\ep}_s )ds+\int^t_0\int_{|z|\leq \ep^{1/\alpha}}\ep^{-1/\alpha}z\widetilde{N}^{2}(ds,dz)\\
&&+\int^t_0\int_{|z|> \ep^{1/\alpha}}\ep^{-1/\alpha}z N^{2}(ds,dz).
\end{eqnarray*}
Using It\^{o} formula and taking expectation on both sides, we get
\begin{eqnarray*}
\EE U(Y^{\ep}_t)=\!\!\!\!\!\!\!\!&&U(y)+\frac{1}{\ep}\EE\int^{t}_{0}\langle f(X^{\ep}_{s},Y^{\ep}_s ),D U(Y^{\ep}_{s})\rangle ds\\
&&+\EE\int^{t}_{0}\int_{|z|\leq\ep^{1/\alpha}}\left[U(Y^{\ep}_{s}+\ep^{-1/\alpha}z)-U(Y^{\ep}_{s})
-\langle D U(Y^{\ep}_s), \ep^{-1/\alpha}z\rangle\right]\nu_2(dz)ds\nonumber\\
&&+\EE\int^{t}_{0}\int_{|z|>\ep^{1/\alpha}}[U(Y^{\ep}_{s}+\ep^{-1/\alpha}z)-U(Y^{\ep}_{s})]\nu_2(dz)ds,
\end{eqnarray*}
which implies
\begin{eqnarray*}
\frac{d\EE U(Y^{\ep}_t)}{dt}=\!\!\!\!\!\!\!\!&&\frac{1}{\ep}\EE\langle f(X^{\ep}_{t},Y^{\ep}_t ),D U(Y^{\ep}_{t})\rangle\\
&&+\EE\int_{|z|\leq\ep^{1/\alpha}}\left[U(Y^{\ep}_{t}+\ep^{-1/\alpha}z)-U(Y^{\ep}_{t})
-\langle D U(Y^{\ep}_t), \ep^{-1/\alpha}z\rangle\right]\nu_2(dz)\nonumber\\
&&+\EE\int_{|z|>\ep^{1/\alpha}}[U(Y^{\ep}_{t}+\ep^{-1/\alpha}z)-U(Y^{\ep}_{t})]\nu_2(dz):=\sum^3_{i=1}J_i(t).
\end{eqnarray*}
For the term $J_1(t)$, by condition \eref{sm}, there exists $\eta>0$ such that
\begin{eqnarray*}
\langle f(X^{\ep}_{t},Y^{\ep}_t ),D U(Y^{\ep}_{t})\rangle=\!\!\!\!\!\!\!\!&&\frac{\langle f(X^{\ep}_{t},Y^{\ep}_t )-f(X^{\ep}_{t},0 ), pY^{\ep}_t\rangle+\langle f(X^{\ep}_{t},0 ), pY^{\ep}_t\rangle}{(|Y^{\ep}_t|^2+1)^{1-p/2}}\\
\leq\!\!\!\!\!\!\!\!&& \frac{-p\beta|Y^{\ep}_t|^2+C_p|Y^{\ep}_t|}{(|Y^{\ep}_t|^2+1)^{1-p/2}}\\
\leq\!\!\!\!\!\!\!\!&& -\eta(|Y^{\ep}_t|^2+1)^{p/2}+C_p.
\end{eqnarray*}
As a consequence,
\begin{eqnarray}
J_1(t)\leq \frac{-\eta\EE U(Y^{\ep}_t)}{\ep}+\frac{C_p}{\ep}.\label{J1}
\end{eqnarray}
For the term $J_2(t)$, by changing variable $y=\ep^{-1/\alpha}z$ and \eref{D^2U}, we obtain
\begin{eqnarray}
J_2(t)\leq \!\!\!\!\!\!\!\!&&\frac{1}{\ep}\EE\int_{|y|\leq 1}\left[U(Y^{\ep}_{t}+y)-U(Y^{\ep}_{t})
-\langle D U(Y^{\ep}_t), y\rangle\right]\nu_2(dy)\nonumber\\
\leq \!\!\!\!\!\!\!\!&&\frac{C_p}{\ep}\EE\int_{|y|\leq 1}|y|^2\nu_2(dy)\leq \frac{C_p}{\ep},\label{J2}
\end{eqnarray}
and by \eref{DU}, we have
\begin{eqnarray}
J_3(t)\leq \!\!\!\!\!\!\!\!&&\frac{1}{\ep}\EE\int_{|y|\geq 1}\left[U(Y^{\ep}_{t}+y)-U(Y^{\ep}_{t})\right]
\nu_2(dy)\nonumber\\
\leq \!\!\!\!\!\!\!\!&&\frac{C_p}{\ep}\EE\int_{|y|\geq 1}\left(|Y^{\ep}_t|^{p-1}+|y|^{p-1}\right)\nu_2(dy)
\leq \frac{\eta\EE U(Y^{\ep}_t)}{2\ep}+\frac{C_p}{\ep}.\label{J3}
\end{eqnarray}
Combining \eref{J1}-\eref{J3}, we get
\begin{eqnarray*}
\frac{d\EE U(Y^{\ep}_t)}{dt}\leq\!\!\!\!\!\!\!\!&&\frac{-\eta\EE U(Y^{\ep}_t)}{2\ep}+\frac{C_p}{\ep}.
\end{eqnarray*}
Now the comparison theorem implies that for any $t\geq 0$,
\begin{eqnarray}
\EE U(Y^{\ep}_t)\leq \!\!\!\!\!\!\!\!&&e^{\frac{-\eta t}{2\ep}}(|y|^2+1)^{p/2}+\frac{C_p}{\ep}\int^t_0 e^{\frac{-(t-s)\eta }{2\ep}}ds
\leq C_{p}(1+|y|^p),\label{F3.11}
\end{eqnarray}
which implies \eref{Y} holds.
Estimate \eref{X} follows immediately from \eref{F3.8} and \eref{F3.11}.
\end{proof}

Concerning the frozen equation (\ref{FEQ}), we have the following result.

\begin{lemma}\label{L3.2}
Assume that $f(x,\cdot)\in C^1_b$ and condition \eref{sm} holds. Then we have for any $1\leq p<\alpha$, $T\geq 1$,
\begin{eqnarray}
\sup_{t\geq 0}\EE|Y_{t}^{x,y}|^p\leq C_p(1+|y|^p),\label{FEq0}
\end{eqnarray}
\begin{eqnarray}
\EE\left(\sup_{t\in [0,T]}|Y_{t}^{x,y}|^p\right)\leq C_p T^{p/\alpha}+|y|^p.\label{FEq}
\end{eqnarray}
\end{lemma}

\begin{proof}

The estimate \eref{FEq0} can be proved easily using the same argument in the proof of \eref{Y}, we omit the details.
Now we prove \eref{FEq}.
For any fixed $T\geq 1$, define 
$$U_T(y):=(|y|^{2}+T^{2/\alpha})^{p/2},$$
which satisfies
\begin{eqnarray}
|DU_T(y)|=\left|\frac{py}{(|y|^2+T^{2/\alpha})^{1-p/2}}\right|\leq C_p|y|^{p-1},\label{DUT}
\end{eqnarray}
\begin{eqnarray}
\|D^2 U_T(y)\|=\left\|\frac{p I_{d_2}}{(|y|^2+T^{2/\alpha})^{1-p/2}}-\frac{p(p-2)y\otimes y}{(|y|^2+T^{2/\alpha})^{2-p/2}}\right\|\leq C_p T^{\frac{2}{\alpha}\left(\frac{p}{2}-1\right)},\label{D^2UT}
\end{eqnarray}
where $I_{d_2}$ is the $d_2\times d_2$ identity matrix.
By It\^o's formula and (\ref{555}), we get
\begin{align}
&U_T(Y^{x,y}_t)=U_T(y)+\int^t_0\frac{ \langle f(x,Y^{x,y}_s ), p Y^{x,y}_{s}\rangle}{(|Y^{x,y}_s|^{2}+T^{2/\alpha})^{1-p/2}}ds\nonumber\\
&+\int^{t}_{0}\int_{|z|\leq T^{1/\alpha}}[U_T(Y^{x,y}_{s-}+z)-U_T(Y^{x,y}_{s-})]\widetilde{N}^{2}(ds,dz)\nonumber\\
&+\int^{t}_{0}\int_{|z|\leq T^{1/\alpha}}\left[U_T(Y^{x,y}_{s}+z)-U_T(Y^{x,y}_{s})
- \frac{\langle p Y^{x,y}_{s}, z\rangle}{(|Y^{x,y}_s|^{2}+T^{2/\alpha})^{1-p/2}}\right]\nu_2(dz)ds\nonumber\\
&+\int^{t}_{0}\int_{|z|>T^{1/\alpha}}[U_T(Y^{x,y}_{s-}+z)-U_T(Y^{x,y}_{s-})]N^{2}(ds,dz)
:=U_T(y)+\sum^{4}_{i=1}I_i(t).\label{I}
\end{align}
For the term $I_1(t)$, the condition \eref{sm} implies
$$
\langle f(x,y), y\rangle\leq \langle f(x,y)-f(x,0), y\rangle+\langle f(x,0), y\rangle\leq -\frac{\beta}{2}|y|^2+C.
$$
It is easy to see
\begin{eqnarray}
\mathbb{E}\left[\sup_{0\leq t\leq T}|I_1(t)|\right]\leq\!\!\!\!\!\!\!\!&& \int^{T}_{0}\frac{C_p}{(|Y^{x,y}_s|^{2}+T^{2/\alpha})^{1-p/2}}ds\leq C_p T^{p/\alpha+1-2/\alpha}.\label{I1}
\end{eqnarray}
For the term $I_2(t)$, by Burkholder-Davis-Gundy's inequality and \eref{DUT}, we have
\begin{eqnarray}
&&\mathbb{E}\left[\sup_{t\in [0,T]}|I_2(t)|\right]
\leq  C\EE\left[\int^T_0\int_{|z|\leq T^{1/\alpha}} \int^1_0 |DU_T(Y^{x,y}_s+z\xi)|^2 d\xi |z|^2 N^2(ds,dz)\right]^{1/2}\nonumber\\
\leq\!\!\!\!\!\!\!\!&& C\EE\left[\int^T_0\int_{|z|\leq T^{1/\alpha}} (|Y^{x,y}_s|^{2p-2}|z|^2+|z|^{2p}) N^2(ds,dz)\right]^{1/2}\nonumber\\
\leq\!\!\!\!\!\!\!\!&& C\EE\left\{\left(\sup_{s\in [0,T]}|Y^{x,y}_s|^{p-1}\right)\left[\int^T_0\int_{|z|\leq T^{1/\alpha}} |z|^2 N^2(ds,dz)\right]^{1/2}\right\}\nonumber\\
&&+C\EE\left[\int^T_0\int_{|z|\leq T^{1/\alpha}} |z|^{2p} N^2(ds,dz)\right]^{1/2}\nonumber\\
\leq\!\!\!\!\!\!\!\!&& \frac{1}{4}\mathbb{E}\left(\sup_{t\in [0,T]}|Y^{x,y}_t|^{p}\right)+C\left[\EE\int^T_0\int_{|z|\leq T^{1/\alpha}} |z|^2 \nu_2(dz)ds\right]^{p/2}\nonumber\\
&&+C\left[\EE\int^T_0\int_{|z|\leq T^{1/\alpha}} |z|^{2p} \nu_2(dz)ds\right]^{1/2}
\leq\frac{1}{4}\mathbb{E}\left(\sup_{t\in [0,T]}|Y^{x,y}_t|^{p}\right)+C_p T^{p/\alpha}.\nonumber
\end{eqnarray}
For the term $I_3(t)$, by Taylor's expansion and \eref{D^2UT}, we have
\begin{eqnarray}
\mathbb{E}\left[\sup_{t\in [0,T]}|I_3(t)|\right]\leq\!\!\!\!\!\!\!\!&& C_p T^{\frac{2}{\alpha}\left(\frac{p}{2}-1\right)}\int^T_0\int_{|z|\leq T^{1/\alpha}} |z|^2 \nu_2(dz)ds\leq C_{p}T^{p/\alpha}.\label{I3}
\end{eqnarray}
For the term $I_4(t)$, by \eref{DUT} again, we obtain
\begin{eqnarray}
\mathbb{E}\left[\sup_{t\in [0,T]}|I_4(t)|\right]
\leq\!\!\!\!\!\!\!\!&& C\EE\int^T_0\int_{|z|>T^{1/\alpha}} |U_T(Y^{x,y}_s+z)-U_T(Y^{x,y}_s)| \nu_2(dz)ds\nonumber\\
\leq\!\!\!\!\!\!\!\!&& C\EE\int^T_0\int_{|z|> T^{1/\alpha}} (|Y^{x,y}_s|^{p-1}|z|+|z|^{p}) \nu_2(dz)ds\nonumber\\
\leq\!\!\!\!\!\!\!\!&& C\EE\left\{\left(\sup_{s\in [0,T]}|Y^{x,y}_s|^{p-1}\right)\left[\int^T_0\int_{|z|> T^{1/\alpha}} |z| \nu_2(dz)ds\right]\right\}\nonumber\\
&&+C\int^T_0\int_{|z|> T^{1/\alpha}} |z|^{p} \nu_2(dz)ds\nonumber\\
\leq\!\!\!\!\!\!\!\!&& \frac{1}{4}\mathbb{E}\left(\sup_{t\in [0,T]}|Y^{x,y}_t|^{p}\right)+C\left[\int^T_0\int_{|z|>T^{1/\alpha}} |z| \nu_2(dz)ds\right]^{p}\nonumber\\
&&+C\int^T_0\int_{|z|> T^{1/\alpha}} |z|^{p} \nu_2(dz)ds\nonumber\\
\leq\!\!\!\!\!\!\!\!&& \frac{1}{4}\mathbb{E}\left(\sup_{t\in [0,T]}|Y^{x,y}_t|^{p}\right)+C_p T^{p/\alpha}.\label{I4}
\end{eqnarray}
Combining \eref{I}-\eref{I4}, we obtian
\begin{eqnarray*}
\mathbb{E}\left(\sup_{t\in [0,T]}|Y^{x,y}_t|^p\right)\leq\!\!\!\!\!\!\!\!&&C_p T^{p/\alpha}+|y|^p.
\end{eqnarray*}
The proof is complete.
\end{proof}

\begin{remark}\label{Re5.3}
Let $\{\tilde{Y}^{\vare}_{t}, t\geq 0\}$ solves
\begin{eqnarray*}
\tilde{Y}^{\vare}_{t}=\!\!\!\!\!\!\!\!&&y+\int^{t}_0f(X^{\vare}_{s\vare}, \tilde{Y}^{\vare}_{s})ds+\tilde{L}^2_{t},
\end{eqnarray*}
where $\tilde{L}^2_{t}$ is given in (\ref{R3.11}).
Using the condition $\sup_{x\in\RR^{d_1}}|f(x,0)|<\infty$ and by the same argument as in the proof of \eref{FEq}, we can   obtain that for any $T\geq 1$,
$$
\EE\left[\sup_{t\in [0,T]}|\tilde{Y}^{\vare}_{t}|^p\right]\leq C_p T^{p/\alpha}+|y|^p.
$$
On the other hand, note that $Y^{\vare}_{t\vare}$ in (\ref{R3.11}) and $\tilde{Y}^{\vare}_{t}$ have the same law. As a consequence, we have
\begin{eqnarray}
\EE\left(\sup_{t\in [0,T]}|Y^{\vare}_{t}|^p\right)=\EE\left(\sup_{t\in [0,T/\vare]}|\tilde{Y}^{\vare}_{t}|^p\right)\leq C_{p,T}\vare^{-p/\alpha}+|y|^p,\quad \forall \varepsilon\in (0,T].\label{Yvare}
\end{eqnarray}
\end{remark}

Concerning the averaged equation, we have the following result.

\begin{lemma} \label{PMA}
Suppose that the assumptions in Proposition \ref{Ergodicity} hold and that $b\in C^{1,1}_b$.
Then for any $x\in\RR^{d_1}$, Eq.(\ref{1.3}) has a unique solution $\bar{X}_t$. Moreover, for any $T>0$, there exists a constant $C_{T}>0$ such that for any $1\leq p<\alpha$,
\begin{eqnarray}
\mathbb{E}\left(\sup_{t\in [0, T]}|\bar{X}_{t}|^{p}\right)\leq C_{T}(1+|x|^{p}).\label{3.9}
\end{eqnarray}
\end{lemma}
\begin{proof}
By Proposition \ref{Ergodicity} and Lemma \ref{L3.6}, for any $t>0$, $x_1,x_2\in\RR^{d_1}$, we have
\begin{eqnarray*}
|\bar{b}(x_1)-\bar{b}(x_2)|\leq\!\!\!\!\!\!\!\!&&\left|\bar b(x_1)-\EE b(x_1, Y^{x_1,0}_t)\right|+\left| \EE b(x_2, Y^{x_2,0}_t)-\bar b(x_2)\right|\\
&&+\EE \left|b(x_1, Y^{x_1,0}_t)-b(x_2,Y^{x_2,0}_t)\right|\\
\leq\!\!\!\!\!\!\!\!&&Ce^{-\frac{\beta t}{2}}+C\left(|x_1-x_2|+\EE|Y^{x_1,0}_t-Y^{x_2,0}_t|\right)
\leq Ce^{-\frac{\beta t}{2}}+C|x_1-x_2|.
\end{eqnarray*}
Letting $t\rightarrow \infty$, we arrive at the Lipschitz continuous property of $\bar{b}$.
Hence by \cite[Theorems 6.2.3 and Theorem 6.2.11]{A}, there exists a unique solution $\{\bar X_t, t\geq 0\}$ to Eq. (\ref{1.3}). Moreover, estimate (\ref{3.9}) can be easily obtained by following a similar argument as in the proof of \eref{X}. The proof is complete.
\end{proof}

\vspace{0.3cm}
\textbf{Acknowledgment}. We would like to thank Professor Renming Song for useful discussion,  and  the referees for  carefully
reading  the manuscript and providing  many  suggestions and comments. This work is supported by the NNSF of China (12090011, 12071186, 11771187, 11931004) and the Project Funded by the Priority Academic Program Development of Jiangsu Higher Education Institutions.

\end{document}